\documentclass[11pt]{article}

\usepackage{latexsym}
\usepackage{amsfonts,amsthm,amsmath,amssymb,mathrsfs,color}
\usepackage{etoolbox}
\usepackage{leftidx}   
\usepackage{authblk}
\usepackage{xcolor}
\usepackage{bm}
\usepackage{cases}

\usepackage{graphicx}

\renewcommand{\qedsymbol}{\rule{1ex}{1ex}}

\date{}

\newtheorem{theorem}{Theorem}[section]
\newtheorem{lemma}{Lemma}[section]
\newtheorem{proposition}{Proposition}[section]
\newtheorem{corollary}{Corollary}[section]
\newtheorem{remark}{Remark}[section]
\newtheorem{definition}{Definition}[section]

\theoremstyle{definition}

\makeatletter 
\@addtoreset{equation}{section}
\makeatother 

\begin{document}

\title{Reliability properties of $k$-out-of-$n$ systems \\  with one cold standby unit}

\author[a]{Anna Dembi\'{n}ska\footnote{Corresponding author, e-mail address: dembinsk@mini.pw.edu.pl }}
\author[b]{Nikolay I. Nikolov}
\author[b]{Eugenia Stoimenova}

\affil[a]{Faculty of Mathematics and Information Science,
Warsaw University of Technology, \newline
ul. Koszykowa 75, 00-662 Warsaw, Poland}
\affil[b]{Institute of Mathematics and Informatics, Bulgarian Academy of Sciences, \newline
Acad. G. Bontchev str., block 8, 1113 Sofia, Bulgaria}


\maketitle

\begin{abstract}
In this paper, we study reliability properties of a $k$-out-of-$n$ system with a single cold standby unit. We mainly focus on the case when the system operates in discrete time. In order to describe its aging behavior we consider three different mean residual life functions. By using some properties of order statistics we present several monotonicity results associated with these reliability characteristics. Since the calculation of the described quantities requires  finding  sums of infinite series, we provide a procedure to approximate them with an error not greater than a~desired  value. As an illustration we consider three special cases when the component lifetimes have geometric, negative binomial and discrete Weibull distributions.
\end{abstract}


\bigskip

\begin{center}
ACRONYMS
\end{center}

RV \; \;  random variable

IID \; \; independent and identically distributed

IFR \; \; increasing failure rate

DFR \; \, decreasing failure rate

\begin{center}
NOTATION
\end{center}

$n$ \; \; \;  number of active components in a system

$X_i$ \; \; \   lifetime of $i$th active component, $i=1,2,\ldots,n$

$F_i(\cdot)$  \; cumulative distribution function of $X_i$

$\bar{F}_i(\cdot)$ \;  reliability function of $X_i$

$p_i(\cdot)$ \hspace{2mm}  probability mass function of $X_i$

$X_{i:n}$ \; \  $i$th smallest lifetime in $X_1, X_2,\ldots, X_n$

$Z$ \; \; \; \  lifetime of standby unit

$\bar{G}(\cdot)$ \;  \; reliability function  of $Z$

$T$ \; \; \; \   lifetime of   $k$-out-of-$n$ system with one standby unit

${\cal P}_s$  \; \; \;   set of  all permutations $(j_1,j_2,\ldots,j_n)$ of $(1,2,\ldots,n)$ satisfying
\[
j_1 < j_2 < \cdots < j_s, j_{s+1}<j_{s+2} < \cdots < j_n
\]

 ${\cal P}_{s,r}$ \;  \;  set of  all permutations $(j_1,j_2,\ldots,j_n)$ of $(1,2,\ldots,n)$ satisfying
\[
j_1 < j_2 < \cdots < j_s, j_{s+1}<j_{s+2} < \cdots < j_r, j_{r+1} <j_{r+2} < \cdots < j_n
\]

\section{Introduction}

A $k$-out-of-$n$ system is a technical structure consisting of $n$ components and functioning if and only if at least $k$ of its components work. Such systems play an important role in the reliability theory. They have various applications in engineering, for example, in the design of servers in internet service or in the design of  engines in the aeronautics and automotive industries. Reliability properties of  $k$-out-of-$n$ systems have been studied extensively in the literature from many perspectives and under different assumptions concerning component failures; see, among others, \cite{barlow1984computing,Eryilmaz18TheNumber18,EryilmazDevrim19,EryilmazAtAl16,gurler2009parallel,
navarro2017comparison,SalehiAsadi12,TAVBairamov15,WangPengXing18,ZhangZhaoMa19,ZhangYang10}.


In this paper, we consider $k$-out-of-$n$ systems equipped with one cold standby unit.  The standby unit is unpowered until  the moment of system failure when it immediately replaces  one of the broken components and is put into operation. The cold standby redundancy is a common design technique for   increasing the reliability of a system, used typically when power consumption is critical. For a~list of references where this technique found application we refer the reader to \cite{WXA12}. 
There have been many works on reliability analysis of systems with cold standby units. 
Some recent contributions to this topic include  \cite{eryilmaz2014study}, where the author studied some systems equipped with one cold standby unit which may be put into operation at the time of the first component failure if the whole system fails at this time. A more general case when the standby unit may be activated at the time of the $s$th failure among the components was considered in \cite{franko2015reliability}. In \cite{eryilmaz2017effectiveness} effects of adding cold standby redundancy to a coherent system at system and component levels were investigated. Reliability properties of coherent systems consisting of $n$ independent components and equipped with $r$ cold standby units were presented in \cite{eryilmaz2018coherent}. \cite{NR18} describes $k$-out-of-$n$ systems  with one cold standby unit in the case when the performance of the components is getting worse with time. 
 In all those above mentioned works it is assumed that components lifetimes are jointly continuously distributed. However, situations when the continuity assumption is not adequate are not uncommon. This is the case when components fail only at moments of shocks that occur over  discrete-time periods or when a system  operates in cycles and its working components have certain probabilities of failure upon each cycle. Then  the lifetime of a component (system) is the number of shocks or cycles the component  (system) survives,  takes values $0,1, 2, \ldots$ and therefore is a discrete random variable (RV). In this setting the reliability study of a system becomes more complicated because, in contrast to the case of jointly continuously  distributed components lifetimes, the probability of ties between components failures is not equal to zero.  
Systems operating in cycles were discussed, for example, in \cite{weiss1962certain, Y70}. Discrete time shock models were considered  in \cite{TankEryilmaz15, Eryilmaz16}, among others. Basic properties of $k$-out-of-$n$ systems  with components having discrete lifetimes can be found in  \cite{ dembinska2018reliability} while a detailed depiction of the distribution of the number of failed components when such a system fails was given in \cite{DAVIESDEMB19}.

The main contribution of the present paper is the description of reliability properties of  $k$-out-of-$n$ systems  with a single cold standby unit when component lifetimes are discretely distributed. We present general results that hold for possibly dependent and heterogeneous components. In Section~\ref{sec3}, we obtain formulas describing the reliability function and expectation of time to failure of such a system. We also give comments about numerical evaluation of the formula for expectation in the case when it involves an infinite summation. In Section \ref{sec4}, we analyze used systems. Given particular information about the state of the system upon inspection at time $t$, we establish conditional reliabilities and means of the corresponding residual lifetimes of the system. We consider three forms of conditioning, namely, that the system is working at time $t$, that none of its components are broken at time $t$, and that the  $k$-out-of-$n$ system  is working at time $t$. Moreover, we show how the formulas for mean residual lifetimes containing infinite sums can be applied for numerical computations. Section \ref{sec5} is devoted to stochastic ordering results. We provide there several preservation theorems giving sufficient conditions for stochastic relations (in the sense of usual stochastic order) between lifetimes and residual lifetimes of two systems. We also prove that, in the case of independently working components, if the active components of the system have increasing (decreasing) failure rates, then the mean residual life function of the system given that none of its components are broken is decreasing (increasing) in time. In Section \ref{ill_exa}, we illustrate the utility of results presented in this paper. For several examples of $k$-out-of-$n$ systems  with a  standby unit, we compute and compare expected lifetimes and the three residual life functions discussed in Section \ref{sec4}. In Section \ref{sec7}, we give concluding remarks and suggest problems for future investigations. Appendices contain technical results and some proofs. In Appendix \ref{sec:App1}, we derive new properties of order statistics that are needed for our developments. Most of the proofs of results given in Sections \ref{sec3} - \ref{sec5} are postponed to Appendix \ref{sec:App2}. Computational details concerning examples from Section  \ref{ill_exa} can be found in  Appendix \ref{sec:App3}.

Even though the principal aim of this paper is to study $k$-out-of-$n$ systems equipped with a~standby unit under the assumption that components lifetimes are discretely distributed, some of the obtained results hold in the more general setting of any joint distribution of components lifetimes (Theorems \ref{s5th1} and \ref{s5th2}) or also in the setting of continuous  components lifetimes (Theorem \ref{s5th4}). It is worth pointing out that Theorem  \ref{s5th4} in particular solves an open problem posed in \cite{eryilmaz2012mean} in the context of continuous components lifetimes whether it is true that the mean residual life function given that  none of its components are broken is decreasing whenever components lifetimes have increasing failure rate -~for details see the end of Section  \ref{sec5}.


Throughout the paper we use the following notation:
\begin{itemize}
 \item $T$ denotes the lifetime of a $k$-out-of-$n$ system with a~standby unit;
  \item $X_1, \ldots, X_n$ are  the lifetimes of the active components (i.e. the lifetimes of the $n$ components of the initial $k$-out-of-$n$ system); moreover, $ F_i(x) = P(X_i\le x)$ and $\bar{F}_i(x) = P(X_i>x)$ are the marginal cumulative distribution function and reliability function, respectively, of $X_i$,  $i=1, \ldots, n$;
  \item $Z$ is the lifetime of the standby unit and $\bar{G}(x)=P(Z>x)$ is its reliability function;
  \item $X_{1:n} \le X_{2:n} \le \cdots \le X_{n:n}$  stand for the order statistics obtained from $X_1, \ldots, X_n$ so that $X_{i:n}$ is the $i$th smallest lifetime in $X_1, \ldots, X_n$.
\end{itemize}
A $k$-out-of-$n$ system works as long as at least $k$ of its components function. Consequently, it fails at the moment of the 
$(n-k+1)$th component failure, $X_{n-k+1:n}$. Immediately after this failure the standby unit is put into operation. In the case of systems working in cycles this means that  the standby unit is activated during the cycle in which the system brakes down  and thanks to this new unit the system can survive this cycle. Finally, the system stops working when the new unit or  more than $(n-k+1)$ of the old components fail. Hence, we have, for $2 \le k \le n$,
\begin{equation}\label{f1}
T= \min(X_{n-k+1:n}+Z, X_{n-k+2:n}).
\end{equation}
Defining $X_{n+1:n}=\infty$ we ensure that (\ref{f1}) is also valid for $k=1$.

We focus on the case when $X_1, \ldots, X_n$ and $Z$ are discrete RVs. Then we set  $F_i(x^{-}) = P(X_i< x)$ and  $p_i(x) = P(X_i= x)$, $ i=1, \ldots, n$.
Within the case of discrete lifetimes of components, we consider  the following three special  subcases.
\begin{enumerate}
 \item Independent case when $X_1, \ldots, X_n,Z$ are independent.
 \item Exchangeable case when  the random vector $(X_1, \ldots, X_n,Z)$ is exchangeable, i.e. when any permutation of $(X_1, \ldots, X_n,Z)$ has the same distribution as $(X_1, \ldots, X_n,Z)$.
 \item IID case when
\begin{equation}
\label{IIDcase}
\begin{array}{c}
 (X_1, \ldots, X_n) \hbox{  and } Z \hbox{  are independent, and } X_1, \ldots, X_n  \hbox{ are independent } \\
\hbox{and identically distributed  (IID)  with } F(x) = P(X_i\le x), \\
  \bar{F}(x) = P(X_i > x), F(x^{-}) = P(X_i< x) \hbox{ and } p(x) = P(X_i= x),  i=1, \ldots, n.
\end{array}
\end{equation}

\end{enumerate}

Moreover,  to simplify notation,  we adopt the convention that
any empty sum equals 0,
$ \sum_{j\in \emptyset} a_j =0$,
 any empty product equals 1,
$ \prod_{j\in \emptyset} a_j =1$, and any empty  intersection is equal to the whole sample space, $ \bigcap_{j\in \emptyset} A_j = \Omega$.

Finally, we say that a RV $Y$ is geometrically distributed with parameter $p\in(0,1)$, and write $Y \sim ge(p)$, if
\[
P(Y=t) = p(1-p)^t, \quad t=0, 1, \ldots.
\]



\section{Distribution of $T$}
\label{sec3}

Let us consider a $k$-out-of-$n$ system with one cold standby unit. We will denote by $X_1, \ldots, X_n$ and $Z$  the random lifetimes of the $n$ active elements of the system and of the standby unit, respectively, by  $X_{1:n} \le X_{2:n} \le \cdots \le X_{n:n}$  the order statistics corresponding to $(X_1, \ldots, X_n)$, and by $T$  the lifetime of the system. Assuming that the RVs $X_i$, $i=1, \ldots, n$, and $Z$ take on only non-negative integers as possible values, we describe the distribution of $T$.

Here, and subsequently, ${\cal P}_r$, $r=0,\ldots,n$, and ${\cal P}_{r,s}$, $0 \leq r <s \leq n$, denote the set of permutations $(j_1,j_2,\ldots,j_n)$ of $(1,2,\ldots,n)$, satisfying
$$j_1 < j_2 < \cdots < j_r, j_{r+1}<j_{r+2} < \cdots < j_n,$$
and
$$j_1 < j_2 < \cdots < j_r, j_{r+1}<j_{r+2} < \cdots < j_s, j_{s+1} <j_{s+2} < \cdots < j_n,$$
respectively. Moreover, it is understood that ${\cal P}_0={\cal P}_n=\{(1,2,\ldots,n)\}$, ${\cal P}_{0,s}={\cal P}_s$ and ${\cal P}_{r,n}={\cal P}_r$.

We first compute the reliability function of  $T$. From (\ref{f1}) we see that $T$ is closely related to order statistics from $(X_1,\ldots,X_n)$. Therefore in our developments we will need some new properties of order statistics, which are derived in Appendix \ref{sec:App1}.
Using subsequently (\ref{f1}), (\ref{s2f6}) and (\ref{s2f10}) we get
\begin{eqnarray}
  P(T>t) &=& \sum_{u=0}^{\infty} P(T>t, X_{n-k+1:n} =u) \nonumber\\
   &=& \sum_{u=0}^{\infty} P(\min(X_{n-k+1:n} +Z,X_{n-k+2:n})>t,  X_{n-k+1:n} =u) \nonumber\\
    &=& \sum_{u=0}^t P(  X_{n-k+1:n} =u, X_{n-k+2:n}>t,  u +Z>t)  + \sum_{u=t+1}^{\infty} P(  X_{n-k+1:n} =u) \nonumber\\
    &=&  \sum_{u=0}^t P(  X_{n-k+1:n} =u, X_{n-k+2:n}>t,  Z>t-u)  + \  P( X_{n-k+1:n} > t) \label{s3f1}\\
    &=&    \sum_{v=0}^{n-k} \sum_{(j_1,\ldots,j_n) \in {\cal P}_{v,n-k+1}}  \sum_{u=0}^{t} P\left(\leftidx{_{(X_1, \ldots, X_n, Z)}^{(j_1,\ldots,j_n)}}A^{u,t}_{v,n-k+1}\right) \nonumber \\
    && +   \sum_{v=0}^{n-k} \sum_{(j_1,\ldots,j_n) \in {\cal P}_{v} } P\left(\leftidx{_{(X_1, \ldots, X_n)}^{(j_1, \ldots,j_n)}} C_{v}^t\right), \label{s3f2}
\end{eqnarray}
where
\begin{eqnarray}
\leftidx{_{(X_1, \ldots, X_n, Z)}^{(j_1,\ldots,j_n)}}A^{u,t}_{v,n-k+1}&=&\left(\bigcap_{\ell=1}^v \{X_{j_\ell} <u \}\right) \cap \left(\bigcap_{\ell=v+1}^{n-k+1} \{X_{j_\ell}=u \} \right) \nonumber \\
&&\cap \left(\bigcap_{\ell=n-k+2}^n \{X_{j_\ell}>t \} \right) \cap \{Z >t-u \}  \label{defA}
\end{eqnarray}
and
\begin{equation}\label{defC}
\leftidx{_{(X_1,\ldots, X_n)}^{(j_1,\ldots,j_n)}}  C_v^t=\left(\bigcap_{\ell=1}^v \{X_{j_\ell} \leq t\} \right) \cap \left( \bigcap_{\ell=v+1}^n \{X_{j_\ell}>t \}\right).
\end{equation}

Note that $P(X_{n-k+1:n}>t)$ is the reliability function of a $k$-out-of-$n$ system without a standby unit. Therefore the increase of reliability of $k$-out-of-$n$ system when a standby unit is added is given by the first term of (\ref{s3f2}), i.e., is equal to
\[
\sum_{v=0}^{n-k} \sum_{(j_1,\ldots,j_n) \in {\cal P}_{v,n-k+1}}  \sum_{u=0}^{t} P\left(\leftidx{_{(X_1, \ldots, X_n, Z)}^{(j_1,\ldots,j_n)}} A^{u,t}_{v,n-k+1}\right).
\]
In the case when we have some information about the dependence structure between $X_1, \ldots, X_n, Z$, application of  (\ref{s2f6I}), (\ref{s2f6E}), (\ref{s2f6IID}), (\ref{s2f10I}) and (\ref{s2f10E}) allows us to rewrite (\ref{s3f2}) in a simpler form. The result is summarized in the following theorem.

\begin{theorem}\label{thReliabilityT}

Let $X_1, \ldots, X_n, Z$ have any joint discrete distribution, $F_i(x) = P(X_i \le x)$,  \linebreak  $\bar{F}_i(x) = P(X_i > x)$ and $p_i(x) = P(X_i = x)$ be the marginal cumulative distribution function, reliability function and probability mass function of $X_i$, respectively, $F_i(x^-) = P(X_i < x)$ and $\bar{G}(x) = P(Z > x)$. If $X_1, \ldots, X_n$, and $Z$ take on only non-negative integers as possible values, the reliability function of $T$ is given by (\ref{s3f2}). Furthermore,

$P(T>t)$
\begin{numcases}{=}
\sum_{v=0}^{n-k} \left\{ \sum_{(j_1,\ldots,j_n) \in {\cal P}_{v,n-k+1}}     \left(\prod_{\ell=n-k+2}^n \bar{F}_{j_\ell}(t) \right) \sum_{u=0}^{t}\left(\prod_{\ell=1}^v F_{j_\ell}(u^-) \right)
\left(\prod_{\ell=v+1}^{n-k+1} p_{j_\ell}(u) \right)\bar{G}(t-u)
\right. \nonumber   \\
 + \left. \sum_{(j_1,\ldots,j_n) \in {\cal P}_{v}}     \left(\prod_{\ell=1}^{v} F_{j_\ell}(t) \right)     \left(\prod_{\ell=v+1}^{n} \bar{F}_{j_\ell}(t) \right)  \right\}
    \hbox{ if $X_1,\ldots,X_n,Z$ are independent,}  \label{s3f2I} \\
n! \sum_{v=0}^{n-k} \frac{1}{v!}  \left\{  \frac{1}{(n-k+1-v)! (k-1)!} \sum_{u=0}^{t}      \right.
P\left(\leftidx{_{(X_1,\ldots, X_n,Z)}^{(1,\ldots,n)}} A_{v,n-k+1}^{u,t}\right) \nonumber\\
\qquad \qquad \qquad  + \left.\frac{1}{(n-v)!} P\left(\leftidx{_{(X_1,\ldots, X_n)}^{(1,\ldots,n)}} C_{v}^{t} \right) \right\}
 \hbox{ if $(X_1,\ldots,X_n,Z)$ is exchangeable.} \label{s3f2E}
\end{numcases}

In particular, in the IID case given by (\ref{IIDcase}), we have
\begin{eqnarray}
   P(T>t) &=& n! \sum_{v=0}^{n-k}  \frac{1}{v!} \left\{  \frac{(\bar{F}(t))^{k-1}}{(n-k+1-v)! (k-1)!} \sum_{u=0}^{t} (F(u^-))^v (p(u))^{n-k+1-v}\bar{G}(t-u) \right. \nonumber\\
     & &\left. \qquad \qquad + \frac{1}{(n-v)!} (F(t))^v  (\bar{F}(t))^{n-v} \right\}.\label{s3f2IID}
 \end{eqnarray}

\end{theorem}

Knowing the reliability function of $T$ we can easily compute the probability mass function of $T$ using the relation
\[
P(T=t) = P(T>t-1) - P(T>t) , \quad t=0,1, \ldots .
\]

It is also of interest to find $E\,T$, the expected lifetime of a $k$-out-of-$n$ system with one standby unit. Formula (\ref{s3f1})  yields
\begin{eqnarray}
E\,T &=&  \sum_{t=0}^{\infty} P(T>t)  \label{s3f3}\\
    &=& \sum_{t=0}^{\infty}  \sum_{u=0}^{t} P(X_{n-k+1:n}=u, X_{n-k+2:n}>t, Z>t-u) +EX_{n-k+1:n}, \label{s3f4}
\end{eqnarray}
where the probabilities $P\left(T>t\right)$, $t=0,1, \ldots,$ are given in Theorem \ref{thReliabilityT}. The last term of (\ref{s3f4}), $EX_{n-k+1:n}$, is the expected lifetime of a $k$-out-of-$n$ systems without a standby unit. Consequently, the expression
\[
\sum_{t=0}^{\infty}  \sum_{u=0}^{t} P(X_{n-k+1:n}=u, X_{n-k+2:n}>t, Z>t-u)
\]
denotes the contribution of the standby unit to the expected lifetime of the system.

For $k=1$, that is in the case of parallel systems, (\ref{s3f4}) can be rewriten as \begin{eqnarray}
E\,T &=& \sum_{t=0}^{\infty}  \sum_{u=0}^{t} P(X_{n:n}=u, Z>t-u) + E\,X_{n:n} \nonumber\\
  &=& \sum_{u=0}^{\infty}\sum_{t=u}^{\infty} P(X_{n:n}=u, Z>t-u)+ E\,X_{n:n} \nonumber\\
  &=& \sum_{u=0}^{\infty}\sum_{t=0}^{\infty} P(X_{n:n}=u, Z>t)+ E\,X_{n:n} \nonumber\\
  &=& \sum_{t=0}^{\infty} P(Z>t)+ E\,X_{n:n} = E\, Z + \,EX_{n:n}. \label{s3f5}
\end{eqnarray}
Relation (\ref{s3f5}) can be also obtained directly from (\ref{f1}) which takes on the form $T=X_{n:n}+Z$ if $k=1$. From  (\ref{s3f5}) it follows that finding the expected lifetime of a parallel system with one standby unit reduces to computing $EX_{n:n}$. Methods of evaluating $EX_{n:n}$ under assumption that $X_1, \ldots, X_n$ are discrete RVs are known, see for example \cite{jeske2004tunable}, \cite{eisenberg2008expectation} and \cite{davies2018computing}.  Yet these methods work well only when the supports of $X_i$'s are finite or when $X_i$'s are geometrically distributed.

A problem also arises if $2 \le k \le n$,  $X_i$'s have infinite supports and we want to use (\ref{s3f3}) to compute $ET$ numerically. Indeed then the sum $\sum_{t=0}^{\infty}$ in (\ref{s3f3}) consists of infinitely many non-zero terms and cannot be evaluated using software.

These difficulties can be overcome by allowing approximate results. Before we give details of obtaining desired approximations we note that $E\,T$ may be infinite and establish conditions guaranteeing the finiteness of $E\,T$.

\begin{theorem}
\label{ETfinite}
Let $X_1, \ldots, X_n,Z$ be non-negative RVs with any joint distribution and $Y$ be a RV with the reliability function defined by
\begin{equation}\label{defRVY}
  P(Y>y) = \max_{i=1, \ldots,n} P(X_i>y).
\end{equation}
\begin{description}
  \item[(a)]  In the case of $k=1$, $E\,T<\infty$ if $E\,Y<\infty$ and $E\,Z<\infty$.
  \item[(b)]  \label{s3f7} In the case of $2 \le k \le n$, $E\,T<\infty$ if $E\,Y<\infty$.
  \end{description}
Moreover, if  $X_1, \ldots, X_n$ are identically distributed, then the condition $EY<\infty$  reduces to
$EX_1<~\infty$.
\end{theorem}


Now we are ready to give approximate formulas for computing $ET$ with an error not exceeding the prefixed accuracy $d>0$. We start with the case when $2\le k \le n$.

\begin{theorem}\label{calculateET}
Let $X_1, \ldots, X_n, Z$ be RVs taking values in  the set of non-negative integers, $\bar{F}_i(t) = P(X_i>t), \ i=1, \ldots, n$, and $Y$ be a RV satisfying (\ref{defRVY}). If $E\, Y< \infty$ then, for $2 \le k \le n$ and any joint distribution of $X_1, \ldots, X_n, Z$, the approximate formula
\begin{equation}\label{appET}
  E\,T \approx \sum_{t=0}^{t_0} P(T>t),
\end{equation}
where $P(T>t)$, $t=0,1, \ldots,$ are given in Theorem \ref{thReliabilityT}, and $t_0$ is so chosen that
\begin{equation}
\label{condt0}
  \sum_{t=t_0+1}^{\infty} \max_{i=1,\ldots,n} \bar{F}_i(t) \le d \left( \sum_{v=0}^{n-k+1} \binom{n}{v}  \right)^{-1},
\end{equation}
 gives an error not greater than the fixed value $d>0$. Moreover, if $(X_1, \ldots, X_n)$ is exchangeable with marginal survival function  $\bar{F} (t)= P(X_i>t), \ i=1, \ldots, n$, then (\ref{condt0}) simplifies to
\begin{equation}\label{condt0E}
  \sum_{t=t_0+1}^{\infty} \ \bar{F}(t) \le d \left( \sum_{v=0}^{n-k+1} \binom{n}{v}  \right)^{-1},
\end{equation}
while if $X_1, \ldots, X_n$ are independent then (\ref{condt0}) can by replaced by
\begin{equation}\label{condt0I}
\sum_{t=t_0+1}^{\infty} \max_{i=1,\ldots,n} \bar{F}_i(t) \le \sqrt[k-1]{d \left( \sum_{v=0}^{n-k+1} \binom{n}{v}  \right)^{-1}}.
\end{equation}
In particular, in the case when $X_1, \ldots, X_n$  are IID with common survival function $ \bar{F}(t) = P(X_i>~t)$, $i=1,\ldots, n$, (\ref{condt0I}) reduces to
\begin{equation}\label{condt0IID}
 \sum_{t=t_0+1}^{\infty} \ \bar{F}(t) \le \sqrt[k-1]{d \left( \sum_{v=0}^{n-k+1} \binom{n}{v}  \right)^{-1}}.
\end{equation}
\end{theorem}


Now we turn to the case of parallel systems $(k=1)$. Then, as already observed, $E\,T= E\,Z+ E\,X_{n:n}$ and the problem arises when $X_i$'s have infinite supports and are not geometrically distributed due to difficulties with finding $E\,X_{n:n}$. Below we propose a method of computing $E\,T$ numerically that works well also in this case.

\begin{theorem}\label{calculateETparallel}
Under the assumptions of Theorem \ref{calculateET}, for $k=1$ and any joint distribution of  $X_1, \ldots, X_n, Z$, the approximate formula
\begin{equation}\label{appETparallel}
  E\,T \approx E\,Z + \sum_{t=0}^{t_0} P(X_{n:n}>t),
\end{equation}
where $P(X_{n:n}>t)$, $t=0,1, \ldots$, are given in Lemma \ref{sec2:Lemma3} with $k=1$, and $t_0$ is so chosen that
\begin{equation}\label{condt0parallel}
      \sum_{t=t_0+1}^{\infty} \max_{i=1, \ldots,n} \bar{F}_{i} (t) \le \frac{d}{2^n-1},
\end{equation}
gives an error not exceeding the fixed value $d>0$.

Moreover, if   $(X_1, \ldots, X_n)$ is exchangeable with marginal survival function $\bar{F}(t) = P(X_i>t)$, $i=1,\ldots,n$, then (\ref{condt0parallel}) reduces to
\[
\sum_{t=t_0+1}^{\infty} \bar{F}(t) \le \frac{d}{2^n-1}.
\]
\end{theorem}



It is worth pointing out that to compute $E\,T$ in the case when $X_1, \ldots, X_n$ are independent and geometrically distributed one can apply a simpler method than that presented above. Indeed,  then exact formulas convenient for numerical evaluations  of $X_{r:n}$, $1\leq r\leq n$,   are known (see \cite{davies2018computing}). Consequently in this case,
for parallel systems with one standby unit one can directly use (\ref{s3f5}) while  for non-parallel systems with a standby unit ($2\le k \le n$)  finding $E\,T$ reduces to evaluating the first term of (\ref{s3f4}). Details of this evaluation are given  in Appendix \ref{sec:App3}, Subsection~\ref{sec6geo}.

\section{Residual lifetimes of used systems}
\label{sec4}

In this section we will study three types of residual lifetimes of used $k$-out-of-$n$ systems with one standby unit:

\begin{enumerate}
  \item the usual residual lifetime of the system that represents its lifetime after time $t$, given that it is still working at time $t$, with survival function
    \begin{equation}\label{URL}
      P(T-t>s|T>t), \ s \ge 0;
    \end{equation}

  \item the  residual lifetime at the system level describing system lifetime after time $t$, given that none of its components are broken at time $t$, with survival function
    \begin{equation}\label{RLatSL}
      P(T-t>s|X_{1:n}>t), \ s \ge 0;
    \end{equation}
  \item  the  residual lifetime of the system given that the $k$-out-of-$n$ system is still working at time $t$, described by the following survival function
    \begin{equation}\label{RLSW}
      P(T-t>s|X_{n-k+1:n}>t), \ s \ge 0.
    \end{equation}
\end{enumerate}

We will also derive formulas describing the three corresponding mean residual life functions and discuss their properties.

\subsection{Usual residual lifetime}
\label{subsection4.1}

The usual residual lifetime of a system with survival function given in (\ref{URL}) provides information about the lifetime of the system after it has attained a certain age $t$, and thus describes the aging behavior of the system. The corresponding expectation, $E\, (T-t|T>t)$, called mean residual life function of the system, is one of the most important characteristics in dynamic reliability analysis. Both residual lifetimes and mean residual life functions have been studied extensively in the literature for various technical systems with diverse types of components and dependence structures between them. Here we only briefly mention that for $k$-out-of-$n$ systems with independent homogeneous components and with one independent standby unit these characteristics were considered by Eryilmaz \cite{eryilmaz2012mean} under the  assumption that the component lifetimes are continuous.

In this section we describe them also for the $k$-out-of-$n$ systems with one standby unit but in the case when the component lifetimes are discrete. We give general results that hold  for possibly heterogeneous and dependent components.

\begin{proposition}\label{ThNone0}
Under the assumptions of Theorem \ref{thReliabilityT}, the usual residual lifetime of a $k$-out-of-$n$ system with one standby unit when $X_1, \ldots, X_n$ and $Z$ denote the lifetimes of the $n$ active elements of the system and of the standby unit, respectively, has the following survival function
   \begin{equation}\label{fURL}
     P(T-t>s|T>t) = \frac{ P(T>t+s)}{ P(T>t)}, \quad s=0,1,\ldots,
   \end{equation}
where $ P(T>t+s)$ and $ P(T>t)$ are given by Theorem \ref{thReliabilityT}.

The mean residual life function has the form
   \begin{equation}\label{fMRLF}
     E\,(T-t|T>t) = \frac{ 1}{ P(T>t)} \sum_{s=t}^{\infty} P(T>s).
   \end{equation}
\end{proposition}

\begin{proof}
Formula (\ref{fURL}) follows immediately from the definition of conditional probability. To prove  (\ref{fMRLF}) we use  (\ref{fURL}) to obtain
\[
 E\,(T-t|T>t) =  \sum_{s=0}^{\infty} P(T-t>s|T>t) =\frac{ 1}{ P(T>t)} \sum_{s=0}^{\infty} P(T>t+s),
\]
which gives (\ref{fMRLF}) .
\end{proof}

Even though (\ref{fMRLF}) holds for any discrete joint distribution of component lifetimes, it cannot be used directly for numerical computations in the case when the supports of $X_i$'s are infinite, because then the sum $\displaystyle \sum_{s=t}^{\infty} P(T>s)$ consists of infinitely many non-zero terms. Similarly as in Section~\ref{sec3}, we propose to overcome this problem by allowing approximate results.

Note that to compute $E\,(T-t|T>t)$ with an accuracy not worse than $d>0$, we can use the approximate formula
\begin{equation}\label{appURL}
   E\,(T-t|T>t) \approx  \frac{ 1}{ P(T>t)} \sum_{s=t}^{t_0} P(T>s),
\end{equation}
where $t_0$ is such that
\begin{equation}\label{t0URL}
     \frac{ 1}{ P(T>t)} \sum_{s=t_0+1}^{\infty} P(T>s) \le d.
\end{equation}
Since (\ref{t0URL}) is just (\ref{s3f11}) with ``$d$" replaced by ``$dP(T>t)$", from the proof of Theorem \ref{calculateET} we see that   rules of finding $t_0$ can be obtained  directly from Theorem \ref{calculateET} provided that $2\le k\le n$.

\begin{proposition}\label{calculateETURL}
  Theorem \ref{calculateET} remains true if we replace (\ref{appET}) by  (\ref{appURL}), and ``$d$" in (\ref{condt0})~--~(\ref{condt0IID}) by ``$dP(T>t)$".
\end{proposition}


The case of $k=1$ requires more effort, because then rules of choosing $t_0$ to be used in (\ref{appURL}) cannot be derived directly from Theorem \ref{calculateETparallel}.

\begin{proposition}\label{calculateETURLparallel}
Let $k=1$, $\bar{G}(z)= P(Z>z)$, $E\,Z <\infty$ and assumptions of Theorem \ref{calculateET} hold. Then, to compute $E\, (T-t|T>t)$ with an accuracy not worse than $d>0$, we can use (\ref{appURL}) with $t_0$ satisfying the following two conditions
\begin{equation}\label{t0condG}
  \sum_{s=[(t_0+1)/2]}^{\infty} \bar{G}(s) \le \frac{1}{4} dP(T>t)
\end{equation}
and
\begin{equation}\label{t0condF}
  \sum_{s=[(t_0+1)/2]}^{\infty}\displaystyle \max_{i=1,\ldots,n}\bar{F}_i(s) \le \frac{dP(T>t)}{4(2^n-1)},
\end{equation}
where $[x]$ denotes the integer part of $x$.
\end{proposition}


\begin{remark}\label{remURL2}
In particular, in the special case when $X_i$, $i=1,\ldots,n$, and $Z$ are geometrically distributed  with $X_i\sim ge(p_i)$, $i=1,\ldots,n$, and $Z\sim ge(g)$, (\ref{t0condG}) and (\ref{t0condF}) simplify to
\[
\left[\frac{t_0+1}{2} \right] \ge \log_{(1-g)} \frac{dgP(T>t)}{4} - 1
\]
and
\[
\left[\frac{t_0+1}{2} \right] \ge \log_{\displaystyle (1-\min_{i=1,\ldots,n} p_i)} \frac{dP(T>t)\displaystyle\min_{i=1,\ldots,n} p_i}{4(2^n-1)} - 1,
\]
respectively.

\end{remark}

\subsection{Residual lifetime at the system level}

Now we consider the residual lifetime of a system given that at time $t$ none of its components were broken, with survival function in (\ref{RLatSL}). The corresponding expectation, $E(T-t|X_{1:n}>t)$, is called the mean residual life  function of a system at the system level.
It has been introduced by Bairamov et al. \cite{BAA02}  and studied for systems with IID continuous component lifetimes by Asadi and  Bayramoglu \cite{asadi2006mean} and Eryilmaz \cite{eryilmaz2012mean}, among  others. In this subsection, we will derive some properties of this residual lifetime and mean residual life function for $k$-out-of-$n$ systems with one standby unit.

Using (\ref{f1}) and Lemma \ref{sec2:lemma1} we obtain, for any joint distribution (not necessarily discrete) of $X_1, \ldots, X_n,Z$,
\begin{eqnarray}
   & & P(T-t>s| X_{1:n}>t) = P(\min(X_{n-k+1:n}+Z,X_{n-k+2:n}) >t+s | X_{1:n}>t) \nonumber \\
   &=& P(\min(X_{n-k+1:n}-t+Z,X_{n-k+2:n}-t) >s | X_{1:n}>t)  \nonumber \\
    &=&  P(\min(X_{n-k+1:n}^{(t)}+Z^{(t)},X_{n-k+2:n}^{(t)}) >s) = P(T^{(t)}>s), \ s\geq 0, \label{nowywzor}
\end{eqnarray}
where $X_{n-k+1:n}^{(t)}$ and $X_{n-k+2:n}^{(t)}$ are the $(n-k+1)$th and the $(n-k+2)$th order statistics corresponding to $X_1^{(t)}, \ldots, X_n^{(t)}$, the random vector  $(X_1^{(t)}, \ldots, X_n^{(t)}, Z^{(t)})$ has the following reliability function
\begin{eqnarray}\label{s2f1}
  \lefteqn{P(X_1^{(t)}>x_1, \ldots, X_n^{(t)}>x_n, Z^{(t)}>z)} \nonumber\\
    &=&  {\displaystyle
\frac{P(X_1>\max(t,x_1+t), \ldots, X_n>\max(t,x_n+t), Z>z)}{P(X_1>t, \ldots, X_n>t)}}, \;  x_1,\ldots,x_ n,  z \in \mathbb{R}, 
\end{eqnarray}
and $T^{(t)}$ denotes the lifetime of a $k$-out-of-$n$ system with one standby unit and with element lifetimes  $X_1^{(t)}, \ldots, X_n^{(t)}, Z^{(t)}$.

Hence we have shown the following result. Its last part is a consequence of Corollary \ref{sec2:cor1}.

\begin{theorem}\label{ThNone}
Let  $X_1, \ldots, X_n, Z$  have an arbitrary joint distribution and the random vector $(X_1^{(t)}, \ldots, X_n^{(t)}, Z^{(t)})$ have the reliability function given in (\ref{s2f1}). Next, let $T$ and $T^{(t)}$ be lifetimes of $k$-out-of-$n$ systems with one standby unit each and with element lifetimes $X_1, \ldots, X_n, Z$ and $X_1^{(t)}, \ldots, X_n^{(t)}, Z^{(t)}$, respectively. Then the conditional distribution of $T-t$ given $X_{1:n}>t$ is the same as the  unconditional distribution of $T^{(t)}$. In particular, we have
\[
P(T-t>s| X_{1:n}>t) = P(T^{(t)}>s), \ s\geq 0,
\]
and
\[
E\,(T-t| X_{1:n}>t) = E\,T^{(t)}.
\]
Moreover, in the case when $X_1, \ldots, X_n, Z$ are independent we get $(X_1^{(t)}, \ldots,X_n^{(t)}, Z^{(t)})= (X_1^{(t)}, \ldots,$ $X_n^{(t)}, Z)$ with $X_i^{(t)}$, $i=1,\ldots,n$, having the reliability function given in (\ref{s2f2}) and the RVs $X_1^{(t)}, \ldots, X_n^{(t)}$, $Z$ being independent.
\end{theorem}

Theorems \ref{ThNone} and \ref{thReliabilityT} allow us to write formulas describing $P(T-t>s| X_{1:n}>t)$ for any discrete distribution of  $(X_1, \ldots, X_n, Z)$.

\begin{theorem}
\label{thRLatSL}
Under the assumptions of Theorem \ref{thReliabilityT}, we have, for $s=0,1,\ldots$,
 \begin{eqnarray*}
 P(T-t>s| X_{1:n}>t) &=& P(T^{(t)}>s)\\
     &=& \sum_{v=0}^{n-k} \left\{ \sum_{(j_1,\ldots,j_n) \in {\cal P}_{v,n-k+1}} \sum_{u=1}^{s}  P\left(\leftidx{_{(X_1^{(t)},\ldots, X_n^{(t)}, Z^{(t)})}^{(j_1,\ldots,j_n)}} A_{v,n-k+1}^{u,s}\right) \right. \\
     && \qquad + \left.   \sum_{(j_1,\ldots,j_n) \in {\cal P}_{v}}   P\left(\leftidx{_{(X_1^{(t)},\ldots, X_n^{(t)})}^{(j_1,\ldots,j_n)}} C_{v}^{s}\right) \right\}
 \end{eqnarray*}
\begin{numcases}{=}
\left(\prod_{\ell=1}^n \bar{F}_{\ell}(t) \right)^{-1}
\sum_{v=0}^{n-k} \left\{ \sum_{(j_1,\ldots,j_n) \in {\cal P}_{v,n-k+1}}     \sum_{u=t+1}^{s+t} \left(\prod_{\ell=1}^v ( F_{j_\ell}(u^-) - F_{j_\ell}(t)) \right)\right.\nonumber\\
\times\left(\prod_{\ell=v+1}^{n-k+1} p_{j_\ell}(u) \right)
 \left(\prod_{\ell=n-k+2}^{n} \bar{F}_{j_\ell}(s+t) \right)
\bar{G}(s+t-u)   \nonumber\\
+\left. \sum_{(j_1,\ldots,j_n) \in {\cal P}_{v}}
\left(\prod_{\ell=1}^v ( F_{j_\ell}(s+t) - F_{j_\ell}(t)) \right)
\left(\prod_{\ell=v+1}^n  \bar{F}_{j_\ell}(s+t)   \right) \right\} \nonumber\\
\qquad \qquad \qquad \qquad \qquad \qquad \qquad \qquad     \hbox{ if $X_1,\ldots,X_n,Z$ are independent,}  \nonumber\\
n! \sum_{v=0}^{n-k} \frac{1}{v!}   \left\{  \frac{1}{(n-k+1-v)! (k-1)!} \sum_{u=1}^{s}      \right.
P\left(\leftidx{_{(X_1^{(t)},\ldots, X_n^{(t)},Z^{(t)})}^{(1,\ldots,n)}} A_{v,n-k+1}^{u,s}\right) \nonumber\\
\qquad \qquad\qquad \qquad+ \left.\frac{1}{(n-v)!} P\left(\leftidx{_{(X_1^{(t)},\ldots, X_n^{(t)})}^{(1,\ldots,n)}} C_{v}^{s} \right) \right\}
 \hbox{ if $(X_1,\ldots,X_n)$ is exchangeable,} \nonumber
\end{numcases}
where the events $\displaystyle \leftidx{_{(X_1^{(t)},\ldots, X_n^{(t)},Z^{(t)})}^{(j_1,\ldots,j_n)}} A_{v,n-k+1}^{u,s}$ and $\displaystyle  \leftidx{_{(X_1^{(t)},\ldots, X_n^{(t)})}^{(j_1,\ldots,j_n)}} C_{v}^{s}$ are given by (\ref{defA}) and (\ref{defC}), respectively.

In particular, in the IID case described in (\ref{IIDcase}), we have

$P(T-t>s|X_{1:n}>t)$
\begin{eqnarray}
 &=& \frac{n!}{(\bar{F}(t))^n} \sum_{v=0}^{n-k}  \frac{1}{v!} \left\{
  \frac{(\bar{F}(t+s))^{k-1}}{(n-k+1-v)! (k-1)!} \sum_{u=t+1}^{s+t} (F(u^-)-F(t))^v (p(u))^{n-k+1-v}\bar{G}(t+s-u) \right. \nonumber\\
     & &\left. \qquad \qquad\qquad \qquad + \frac{1}{(n-v)!} (F(s+t)-F(t))^v  (\bar{F}(s+t))^{n-v} \right\}.\nonumber
 \end{eqnarray}

\end{theorem}

We are now able to give formulas for $E(T-t|X_{1:n}>t)$ and present methods of their numerical evaluation.

We first consider the case of $2\le k \le n$. We get
\[
E(T-t|X_{1:n}>t) = \sum_{s=0}^{\infty} P(T-t>s|X_{1:n}>t),
\]
where the probabilities $P(T-t>s|X_{1:n}>t)$, $s=0,1,\ldots$, are given in Theorem \ref{thRLatSL}. To compute $E(T-t|X_{1:n}>t)$ numerically with an error not greater than $d>0$, we can apply the approximate formula
\begin{equation}\label{appETatSL}
  E(T-t|X_{1:n}>t) \approx \sum_{s=0}^{s_0} P(T-t>s|X_{1:n}>t),
\end{equation}
where $s_0$ is such that
\[
\sum_{s=s_0+1}^{\infty} P(T-t>s|X_{1:n}>t)\le d.
\]
Theorem \ref{ThNone} yields
\[
\sum_{s=s_0+1}^{\infty}P(T>t+s|X_{1:n}>t) = \sum_{s=s_0+1}^{\infty}P(T^{(t)}>s),
\]
and combined with Theorem \ref{ETfinite} (b) shows that $E(T-t|X_{1:n}>t)$ is finite if $\displaystyle \sum_{s=0}^{\infty} \max_{i=1,\ldots,n} P(X_i^{(t)}>~s~)$ is so. If the later condition is satisfied, then from the proof of Theorem \ref{calculateET}  we see that  $\displaystyle \sum_{s=s_0+1}^{\infty}P(T^{(t)}>~s~)$ $\le d$ provided that
\begin{equation}\label{condt0atSL}
  \sum_{s=s_0+1}^{\infty} \max_{i=1,\ldots,n} P(X_i^{(t)}>s) \le d \left(\sum_{v=0}^{n-k+1} \binom{n}{v} \right)^{-1}.
\end{equation}
Hence to find $s_0$ to be used in (\ref{appETatSL}) we apply (\ref{condt0atSL}). Moreover, in the case when $X_1, \ldots, X_n,Z$ are independent, (\ref{condt0atSL}) can be replaced by
\begin{equation}\label{condETatSLI}
  \sum_{s=s_0}^{\infty} \max_{i=1,\ldots,n} \frac{\bar{F}_i(s+t+1)}{\bar{F}_i(t)} \le \sqrt[k-1]{d \left(\sum_{v=0}^{n-k+1} \binom{n}{v} \right)^{-1} },
\end{equation}
and if additionally $X_1, \ldots, X_n$ are identically distributed with common survivor function $\bar{F}(t) = P(X_i>t)$, $i=1,\ldots,n$, then (\ref{condETatSLI}) simplifies to
\[
\sum_{s=s_0}^{\infty} \bar{F}(s+t+1) \le \bar{F}(t) \sqrt[k-1]{d \left(\sum_{v=0}^{n-k+1} \binom{n}{v} \right)^{-1} }.
\]

We now turn to the case of $k=1$. From Theorem \ref{ThNone}, (\ref{s3f5}) and Theorem \ref{calculateETparallel} we immediately obtain
\begin{equation}
 E(T-t|X_{1:n}>t) = E(T^{(t)}) =  E Z^{(t)} +  E\,X_{n:n}^{(t)} 
    \approx  E \,Z^{(t)} +\sum_{s=0}^{s_0} P (X_{n:n}^{(t)}>s), 
\label{appETatSLparallel}
\end{equation}
where the probabilities $P (X_{n:n}^{(t)}>s)$, $s=0,1,\ldots$, are given by Lemma \ref{sec2:Lemma3} with $k=1$ and $(X_1, \ldots, X_n)$ replaced by $(X_1^{(t)}, \ldots, X_n^{(t)})$. If $\displaystyle \sum_{s=0}^{\infty} \max_{i=1,\ldots,n}  P( X_i^{(t)}>s)<\infty$, then $E\, X_{n:n}^{(t)} <\infty$ and to guarantee that the error of the approximation in (\ref{appETatSLparallel}) does not exceed the fixed value $d>0$, it suffices to choose $s_0$ such that
\begin{equation}\label{condt0atSLparallel}
  \sum_{s=s_0+1}^{\infty} \max_{i=1,\ldots,n}  P( X_i^{(t)}>s) \le \frac{d}{2^n-1}.
\end{equation}
Moreover, if $X_1, \ldots, X_n,Z $ are independent, then (\ref{appETatSLparallel}) and (\ref{condt0atSLparallel}) can be replaced by
$$
  E(T-t|X_{1:n}>t) = E\, Z  +  E\, (X_{n:n}^{(t)})\\
    \approx  E \,Z +\sum_{s=0}^{s_0} P (X_{n:n}^{(t)}>s),
$$
and
\[
\sum_{s=s_0}^{\infty} \max_{i=1,\ldots,n}  \frac{\bar{F}_i(s+t+1)}{\bar{F}_i(t)} \le \frac{d}{2^n-1},
\]
respectively.

The last theorem of this section gives an interesting property of systems consisting of independently working components and such that lifetimes of active components are geometric.

\begin{theorem}\label{geometric}
Let us consider a $k$-out-of-$n$ system with one standby unit and let $X_1, \ldots, X_n,Z$ and $T$ represent the lifetimes of the $n$ active components, the standby unit and the whole system, respectively. If  the lifetimes $X_1, \ldots, X_n,Z$ are independent and $X_1, \ldots, X_n$ are geometrically distributed, then, for $t=0,1, \ldots$, the conditional distribution of $T-t-1$ given $X_{1:n}>t$ is the same as the unconditional distribution of $T$. In particular, for $t=0,1, \ldots$,
\begin{equation}\label{thgf1}
  P(T-(t+1)>s | X_{1:n}>t) =P(T>s), \quad s=-1,0,1, \ldots,
\end{equation}
and
\begin{equation}\label{thgf2}
  E(T-(t+1) | X_{1:n}>t) =E\, T.
\end{equation}

\end{theorem}

\begin{proof}
We follow the notation used in  Theorem \ref{ThNone}. Applying (\ref{nowywzor}) we get
$$
P(T-(t+1)>s| X_{1:n}>t)=P(\min(X_{n-k+1:n}^{(t)}-1+Z^{(t)},X_{n-k+2:n}^{(t)}-1) >s) = P(T^{(t)\star}>s), 
$$
where $T^{(t)\star}$ denotes the lifetime of a $k$-out-of-$n$ system with one standby unit and with element lifetimes  $X_1^{(t)}-1, \ldots, X_n^{(t)}-1, Z^{(t)}$.
Let $X_i$ be geometrically distributed with parameter $p_i \in (0,1)$, $i=1, \ldots, n$. Then by 
 Theorem \ref{ThNone}, for $i=1, \ldots, n$,\ $t=0,1, \ldots$, and $s=-1,0,1, \ldots,$
\[
P(X_i^{(t)}-1>s) = \frac{P(X_i>s+t+1)}{P(X_i>t)} = (1-p_i)^{s+1},
\]
which means that $X_i^{(t)}-1$ has the same distribution as $X_i$.
Consequently,
\[
(X_1^{(t)}-1,\ldots, X_n^{(t)}-1, Z^{(t)}) \stackrel{d}{=} (X_1\ldots, X_n, Z),
\]
where $\stackrel{d}{=}$ stands for equality in distribution, and the conclusion of the theorem follows.
\end{proof}

Note that (\ref{thgf1}) and (\ref{thgf2}) can be rewritten as
\[
P(T-t>s| X_{1:n}\ge t) = P(T>s) \quad s=-1,0,1, \ldots,
\]
and
\[
E\,(T-t| X_{1:n}\ge t) = ET,
\]
where $t=1,2, \ldots$.
Hence Theorem \ref{geometric} asserts that, in the case of independent and geometrically distributed lifetimes of active components independent of the lifetime of the standby unit, the residual lifetime of the system after time $t$, given that all of its elements survived up to (but not necessarily including) time $t$, is the same as the lifetime of a new system. Accordingly, the corresponding mean residual life function is constant and equal to the expected lifetime of a~new system. Clearly this is a consequence of the memoryless property of geometric distribution. We see that this property is carried over to the residual lifetime of the whole system given that all its elements survived. Yet this is not the case for the residual lifetime given that  the whole system survived  nor for the residual lifetime given that the $k$-out-of-$n$ system is working; see Figure~\ref{fig1} in Section~\ref{ill_exa}, where the corresponding mean residual life functions are not constant, even thought the lifetimes of all components are independent and geometrically distributed.

\subsection{Residual lifetime given that the $k$-out-of-$n$ system is working}
\label{subsection4.3}

Eryilmaz \cite{eryilmaz2012mean}, studying dynamic reliability properties of $k$-out-of-$n$ systems with one standby unit, proposed to analyze residual lifetimes of such systems with survival function (\ref{RLSW}), i.e., residual lifetimes under the condition that the  $k$-out-of-$n$ system is still working at a given time $t$. He presented results describing this residual lifetime and its mean in the case when the component lifetimes are independent and continuously distributed and the active components are homogeneous. In this subsection we give analogous results for the case of discretely distributed lifetimes of components. As in previous subsections, we derive our results for the most general setting of possibly dependent and  nonidentical component lifetimes.

\begin{theorem}
Let the assumptions of Theorem \ref{thReliabilityT} hold and consider a $k$-out-of-$n$ system with one standby unit where $X_i$, $i=1,\ldots,n$, $Z$ and $T$ represent the lifetimes of the $n$ active elements, the standby unit and the whole system, respectively. Then the residual lifetime of this system, given that  the $k$-out-of-$n$ system is still working at time  $t$, has the following survival function
\begin{eqnarray}
\lefteqn{  P(T-t>s | X_{n-k+1:n}>t) = \frac{\displaystyle\sum_{u=t+1}^{t+s} h_{k,n}(t+s,u) + P(X_{n-k+1:n}>s+t)  }{P(X_{n-k+1:n}>t)}} \label{fRLSW}\\
    &=&   \frac{\displaystyle\sum_{u=t+1}^{t+s} h_{k,n}(t+s,u) }{P(X_{n-k+1:n}>t)} + P(X_{n-k+1:n}-t>s|X_{n-k+1:n}>t), \quad s,t=0,1\ldots. \label{fRLSW1}
\end{eqnarray}
where the function
\begin{equation}
\label{s2f3}
  h_{k,n}(t,u) = P(X_{n-k+1:n}=u, X_{n-k+2:n}>t, Z>t-u), \quad 1\le k \le n, u\leq t,
\end{equation}
 is described by Lemma \ref{sec2:lemma2}, and the probabilities $P(X_{n-k+1:n}>t)$,  $P(X_{n-k+1:n}>s+t)$ are given by Lemma \ref{sec2:Lemma3}. The corresponding mean residual life function has the form
\begin{equation}\label{fMRLFSW}
 E\,(T-t| X_{n-k+1:n}>t) = \frac{1}{P(X_{n-k+1:n}>t)} \sum_{s=t}^{\infty} \left(  \sum_{u=t+1}^{s} h_{k,n}(s,u) + P(X_{n-k+1:n}>s)   \right).
\end{equation}

\end{theorem}

\begin{proof}
We have, for $s,t = 0,1, \ldots$,
\begin{equation}\label{s4.3f1}
P(T-t>s | X_{n-k+1:n}>t) =  \frac{ P(T>t+s,X_{n-k+1:n}> t)  }{P(X_{n-k+1:n}>t)}.
\end{equation}
But
\begin{eqnarray}
 P(T>t+s , X_{n-k+1:n}>t)&=&  \sum_{u=t+1}^{\infty} P(T>t+s , X_{n-k+1:n}=u) \nonumber\\
    &=& \sum_{u=t+1}^{\infty} P(X_{n-k+1:n}=u, X_{n-k+2:n}>t+s, Z>t+s-u  ) \nonumber\\
    &=&  \sum_{u=t+1}^{t+s} h_{k,n}(t+s,u) + \sum_{u=t+s+1}^{\infty} P(X_{n-k+1:n}=u). \label{s4.3f2}
\end{eqnarray}
 Combining  (\ref{s4.3f1}) and (\ref{s4.3f2}) gives (\ref{fRLSW}) and (\ref{fRLSW1}).
Formula (\ref{fRLSW}) yields (\ref{fMRLFSW}).
\end{proof}

The second term in (\ref{fRLSW1}) is the residual lifetime of a $k$-out-of-$n$ system without a standby unit under the condition that this system is still working at time $t$. Therefore the first term in (\ref{fRLSW1}) represents the increase of the residual lifetime of $k$-out-of-$n$ system, given that it survived beyond time $t$, when one standby unit is added to this system.

To evaluate (\ref{fMRLFSW}) numerically we may  need to use the approximate formula
\begin{equation}\label{appRLSW}
 E\,(T-t| X_{n-k+1:n}>t) \approx  \frac{1}{P(X_{n-k+1:n}>t)} \sum_{s=t}^{t_0} \left(  \sum_{u=t+1}^{s} h_{k,n}(s,u) + P(X_{n-k+1:n}>s)   \right).
\end{equation}
To ensure that (\ref{appRLSW}) gives an error not greater than the fixed value $d>0$ we choose $t_0$ satisfying
\begin{equation}\label{s4.3f3}
 \sum_{s=t_0+1}^{\infty}\left( \sum_{u=t+1}^{s} h_{k,n}(s,u) + P(X_{n-k+1:n}>s)\right)  \le d P(X_{n-k+1:n}>t).
\end{equation}
From (\ref{fRLSW})
\[
\sum_{u=t+1}^{s}  h_{k,n}(s,u) + P(X_{n-k+1:n}>s)
 = P(T>s, X_{n-k+1:n}>t) \le P(T>s)
\]
and consequently (\ref{s4.3f3}) holds if
\begin{equation}\label{s4.3f4}
  \sum_{s=t_0+1}^{\infty}  P(T>s) \le d P(X_{n-k+1:n}>t).
\end{equation}

Comparing (\ref{s4.3f4}) and (\ref{t0URL}) and using the same arguments as in Subsection \ref{subsection4.1}  we obtain the following results.

\begin{proposition}\label{calculateETRLSW}
 Theorem \ref{calculateET} remains true if we replace (\ref{appET}) by (\ref{appRLSW}),  and ``$d$" in (\ref{condt0}) -- (\ref{condt0IID}) by ``$d P(X_{n-k+1:n}>t)$".
\end{proposition}

\begin{proposition}\label{calculateETRLSWparallel}
Proposition \ref{calculateETURLparallel} remains true if we replace ``$E\,(T-t|T>t)$" by ``$E\,(T-~t|X_{n:n}>~t~)$", (\ref{appURL}) by (\ref{appRLSW}) with $k=1$, and ``$dP(T>t)$" in (\ref{t0condG}) and (\ref{t0condF}) by ``$dP(X_{n:n}>t)$".

\end{proposition}



\section{Monotonicity properties}
\label{sec5}

First we recall definitions and basic properties of concepts that we will use to explore aging behavior of $k$-out-of-$n$ systems with one standby unit and to compare lifetimes of such systems with different components.

Throughout this section increasing (decreasing) means non-decreasing (non-increasing). For the RV $X$ and  event $A$ with positive probability by $[X|A]$ we denote any RV having the same distribution as the conditional distribution of $X$ given $A$, i.e., for any $s$,
\[
P([X|A]>s) = P(X>s|A).
\]
Furthermore, for $a>0$, $\frac{a}{0}$ is defined to be equal to $\infty$.

\begin{definition}

Let $X$ and $Y$ be two RVs. We say that
\begin{itemize}
  \item $X$ is smaller than $Y$ in the usual stochastic order (and write $X \le_{st} Y$) if, for all $s$,
\[
P(X>s) \le P(Y >s);
\]
  \item $X$ is smaller than $Y$ in the hazard rate order (and write $X \le_{hr} Y$) if
\[
\frac{P(Y>t)}{P(X >t)} \mbox{ increases in $t$ for $t$ such that  $P(X >t)>0$ or $P(Y >t)>0$.}
\]
\end{itemize}

\end{definition}

\begin{lemma}\label{propUSO}

Let $X_1, \ldots, X_n$ be a set of independent RVs and $Y_1,\ldots, Y_n$ be another set of independent RVs. If, for $i=1, \ldots, n$,
\[
X_i \le_{st} Y_i
\]
and the function $\varphi:\ \mathbb{R}^n \rightarrow \mathbb{R}$ is componentwise increasing, then
\[
\varphi(X_1, \ldots, X_n) \le_{st} \varphi(Y_1,\ldots, Y_n).
\]

\end{lemma}

\begin{lemma}\label{propHRO}
For two RVs $X$ and $Y$ with right endpoints of supports $r_X$ and $r_Y$, respectively, we have $X \le_{hr} Y$ if and only if (iff) $r_X\le r_Y$ and for   $t$ such that $P(X>t)>0$ and all $s$, $P(X-t>s|X>t) \le P(Y-t>s|Y>t)$, or equivalently iff  $r_X\le r_Y$ and, for $t$ such that $P(X>t)>0$, $[X-t|X>t] \le_{st} [Y-t|Y>t]$.

\end{lemma}

Let us also recall that $X \le_{hr} Y$ implies $X \le_{st} Y$, and if  $X \le_{st} Y$ then $EX \le EY$, provided the expectations exist.

\begin{definition}
Let $X$ be a non-negative RV with a  continuous distribution with support being an interval or with a discrete distribution with support $\{ 0,1,\ldots, N \}$, $N\le \infty$. Denote $supp(X)$ the support of $X$. We say that $X$ has increasing (decreasing) failure rate distribution (denoted by IFR (DFR) distribution) if for any $t_1\le t_2$ such that $t_1, t_2\in supp(X)$ and $P(X>t_2)>0$, we have
\[
P(X-t_1>s|X>t_1) \ge (\le) P(X-t_2>s|X>t_2), \quad s\in \mathbb{R},
\]
or equivalently
\[
[X-t_1|X>t_1] \ge_{st} (\le_{st}) [X-t_2|X>t_2].
\]
\end{definition}

For more details on usual stochastic order, hazard rate order and IFR (DFR) distributions we refer the reader to \cite{belzunce2015introduction} and  \cite{shaked2007stochastic}. Comparisons of technical systems based on stochastic orders can be found among others in \cite{navarro2017comparison, khaledi2007ordering, navarro2011coherent, navarro2012comparisons, belzunce2001partial, li2006some} and in the references given there.

Now we are ready to state and prove results concerning stochastic orderings of lifetimes and residual lifetimes of $k$-out-of-$n$ systems with one standby unit in the situation when they are build from different components.

In Theorems \ref{s5th1} - \ref{s5th3} given below we use the following notation:
\begin{itemize}
  \item $S_1$ and $S_2$ are two $k$-out-of-$n$ systems with one standby unit each;
  \item $T_i$ denotes the lifetime of system $S_i$, $i=1,2$;
  \item $X_1, \ldots, X_n$  and $Y_1,\ldots, Y_n$ are lifetimes of active components of systems $S_1$ and $S_2$, respectively;
  \item $Z_i$ is the lifetime of standby unit in system $S_i$, $i=1,2$.
\end{itemize}

Lemma \ref{propUSO} implies the following result.

\begin{theorem}
\label{s5th1}
If $X_1, \ldots, X_n, Z_1$ are independent, $Y_1,\ldots, Y_n, Z_2$ are independent and $X_i \le_{st} Y_i$, $i=1, \ldots, n$, $Z_1 \le_{st} Z_2$, then
\begin{equation}\label{s5th1assert}
T_1 \le_{st} T_2 \ \mbox{and, in consequence, } \ ET_1 \le ET_2.
\end{equation}
\end{theorem}

\begin{proof}
Let $\varphi_k:\ \mathbb{R}^{n+1} \rightarrow \mathbb{R}$, $k=1, \ldots, n$, be functions defined by
\[
\varphi_k(x_1, \ldots, x_n, z)= \left\{
\begin{array}{ll}
  z+x_{n-k+1:n} & \mbox{if }\ k=1, \\
  \min(x_{n-k+2:n}, z+x_{n-k+1:n} ) & \mbox{if }\ 2\le k\le n,
\end{array}
\right.
\]
where $x_{1:n} \le x_{2:n} \le \cdots \le x_{n:n}$ are obtained after rearranging $x_1, \ldots, x_n$ in order of magnitude. Clearly, the functions $\varphi_k$, $k=1, \ldots, n$, are componentwise  increasing. Using Lemma \ref{propUSO} we get
\[
\varphi_k(X_1, \ldots, X_n, Z_1) \le_{st}  \varphi_k(Y_1, \ldots, Y_n, Z_2),
\]
which is equivalent to $T_1 \le_{st} T_2$.
\end{proof}

Theorem \ref{s5th1} is very general. It holds not only in the setup of discrete lifetimes of components. It can also be applied in the case of continuous component lifetimes or even in the case of non-continuous and non-discrete lifetime distributions. It is also worth pointing out that some very special cases of Theorem \ref{s5th1} were given in \cite[Proposition 2(a) and 3]{eryilmaz2012mean}.

To obtain an analogue of Theorem \ref{s5th1} for residual lifetimes at the system level we need stronger assumptions.

\begin{theorem}\label{s5th2}
Under the assumptions of Theorem \ref{s5th1} with ``$X_i \le_{st} Y_i$'' replaced by ``$X_i \le_{hr} Y_i$'', we have, for any $t$ such that $P(X_{1:n}>t)>0$,
\begin{equation}\label{s5th2assert1}
[T_1| X_{1:n}>t] \le_{st} [T_2| Y_{1:n}>t]
\end{equation}
and, in consequence,
\begin{equation}\label{s5th2assert2}
E(T_1-t| X_{1:n}>t) \le  E(T_2-t| Y_{1:n}>t).
\end{equation}
\end{theorem}


Similarly as  Theorem \ref{s5th1},  Theorem \ref{s5th2} is very general, because it holds for any types of component lifetime distributions. In particular, applying it in the case when these distributions are continuous, we obtain a wide extension of \cite[Proposition 2(c)]{eryilmaz2012mean}.

The next theorem concerns only the case of discrete lifetimes of components. We quit the assumption that the lifetimes of active elements are independent, but the price we pay for this is that now we compare $k$-out-of-$n$ systems equipped with a standby unit which differ only in the distributions of lifetimes of the standby units.

\begin{theorem}\label{s5th3}
Let $X_1, \ldots, X_n, Z_1, Z_2$ be discrete RVs taking on only non-negative integers as possible values.
If   $(X_1, \ldots, X_n) \stackrel{d}{=} (Y_1, \ldots, Y_n)$, the random vector $(X_1, \ldots, X_n)$ and  RV $Z_1$ are independent, the random vector $(Y_1, \ldots, Y_n)$ and  RV $Z_2$ are independent, and $Z_1 \le_{st} Z_2$, then

\begin{description}
  \item[(a)]  (\ref{s5th1assert}) holds;
  \item[(b)] for any $t$ such that $P(X_{1:n}>t)>0$, (\ref{s5th2assert1}) and (\ref{s5th2assert2}) hold;
  \item[(c)] for any $t$ such that $P(X_{n-k+1:n}>t)>0$,
  \[
  [T_1| X_{n-k+1:n}>t] \le_{st} [T_2| Y_{n-k+1:n}>t]
  \]
  and, in consequence,
    \[
  E(T_1-t| X_{n-k+1:n}>t) \le  E(T_2-t| Y_{n-k+1:n}>t).
  \]
\end{description}
 \end{theorem}


The following theorem shows that the IFR (DFR) property of active components is inherited   by the residual lifetime of the system at the system level.

\begin{theorem}\label{s5th4}
  Let $X_1, \ldots, X_n, Z$ be independent and non-negative RVs. Assume that $X_1, \ldots, X_n$
are continuously  distributed and have the same support being an interval or are discretely  distributed with  support $\{ 0,1,\ldots, N\}$, $N\le \infty$.
Next, let  $supp(X)$ denote the common support of $X_i$, $i=1, \ldots, n$, and set
${\cal T} = \{ t \in supp(X) \mbox{ such that } P(X_{1:n}>t )>0 \} $.
  If $X_1, \ldots, X_n$ have IFR (DFR) distributions, then $[T-t| X_{1:n}>t]$ is stochastically decreasing (increasing) in $t\in {\cal T}$ in the sense that, for any $t_1<t_2$,  $t_1, t_2\in {\cal T}$, we have
  \begin{equation}\label{s5f5}
      [T-t_1| X_{1:n}>t_1] \ge_{st} (\le_{st} )[T- t_2| X_{1:n}>t_2],
  \end{equation}
 which implies that $E(T-t|X_{1:n}>t)$ is decreasing   (increasing) in $t\in {\cal T}$.
 \end{theorem}


Theorem \ref{s5th4} solves one of the open problems posed in \cite{eryilmaz2012mean} - it shows that indeed the mean residual life function at the system level is decreasing (increasing) when the component lifetimes are IFR (DFR). In fact it suffices to assume that only the lifetimes of active elements are IFR (DFR). It is worth noting that examples presented in the next section reveal that an analogous property in general does not hold nor for the usual residual lifetime nor for the residual lifetime  given that the $k$-out-of-$n$ system is still working.


\section{Illustrative examples}
\label{ill_exa}

In this section, we illustrate the application of the new results obtained in this paper. We fix a level of accuracy $d$ ($d=0.0001$ for values given in tables and $d=0.001$ for that presented in figures) and compute mean times to failures and the three types of residual lifetimes discussed in Section~\ref{sec4} for  some special cases of $k$-out-of-$n$ systems with one standby unit. For the component lifetimes we choose geometric $ge(p)$,  negative binomial $NB(r,p)$ and discrete Weibull  $W(q,\beta)$ distributions.
These three distributions are  of special interest because they can be considered as discrete analogues of the most popular continuous lifetime distributions in reliability analysis: the exponential, gamma and Weibull distributions. Moreover, they will allow us to illustrate well the results obtained in the paper and to draw some new conclusions.
 Definitions of the  negative binomial  and discrete Weibull  distributions together with their properties relevant to our considerations are recalled in Subsections \ref{sec6nb} and \ref{sec6dW}. Moreover, in Appendix \ref{sec:App3} we present computational details, in particular we give explicit formulas for finding numerically  the number of terms that should be added to obtain the desired accuracy $d$.

In Tables \ref{tab1} - \ref{tab4} we provide mean times to failures $E\,T$  of some $k$-out-of-$n$  systems  with a~standby unit.   In all tables, the lifetimes of the system components, $X_{1},\ldots,X_{n},Z$ are assumed to be  independent RVs. For  systems given in Table~\ref{tab1}, $X_{1},\ldots,X_{n},Z$ are geometrically distributed with $X_{i} \sim ge(p)$ for $i=1,\ldots,n$, and $Z \sim ge(g)$. In Table~\ref{tab2}, $X_{i} \sim NB(2,p)$ for $i=1,\ldots,n$ and $Z \sim NB(2,g)$. The values of $E\,T$ in Table~\ref{tab3} are computed for systems with $X_{i} \sim W(q,2)$ $i=1,\ldots,n$ and $Z \sim W(q_{z},2)$. The component lifetimes of  systems in Table~\ref{tab4} are such that $X_{i} \sim W(q,2)$ for $i=1,\ldots,n$ and $Z \sim ge(g)$. For comparative purposes in  Tables \ref{tab1} - \ref{tab4} we also present values of $E(X_{n-k+1:n})$, i.e. mean times to failures of the corresponding  $k$-out-of-$n$  systems  without a standby unit.

\begin{table}[h]

    \begin{minipage}{.45\textwidth}
    \centering
    \begin{tabular}{cccccc}
        \hline
          $p$ & $g$ & $n$ & $k$ & $E(T)$ & $E(X_{n-k+1:n})$ \\ \hline

            $0.25$ & $0.25$ & $3$ & $2$ & $3.8869$ & $2.3977$ \\
            $$ & $$ & $5$ & $2$ & $5.4506$ & $3.9608$ \\
            $$ & $$ & $5$ & $3$ & $3.2086$ & $2.2213$ \\
            $$ & $$ & $10$ & $3$ & $5.4536$ & $4.4672$ \\ [6pt]
            $0.25$ & $0.10$ & $3$ & $2$ & $4.8034$ & $2.3977$ \\
            $$ & $$ & $5$ & $2$ & $6.3674$ & $3.9608$ \\
            $$ & $$ & $5$ & $3$ & $3.6085$ & $2.2213$ \\
            $$ & $$ & $10$ & $3$ & $5.8532$ & $4.4672$ \\ \hline

   \end{tabular}
    \caption{Values of $E(T)$ for components with geometric lifetimes}
    \label{tab1}
    \end{minipage}
    \hspace{0.05\textwidth}
    \begin{minipage}{.45\textwidth}
    \centering
    \begin{tabular}{cccccc}
        \hline
          $p$ & $g$ & $n$ & $k$ & $E(T)$ & $E(X_{n-k+1:n})$ \\ \hline

        $0.25$ & $0.25$ & $3$ & $2$ & $8.2980$ & $5.3781$ \\
        $$ & $$ & $5$ & $2$ & $10.5255$ & $7.7281$ \\
        $$ & $$ & $5$ & $3$ & $7.1103$ & $5.1947$ \\
        $$ & $$ & $10$ & $3$ & $10.2627$ & $8.4974$ \\ [6pt]
        $0.25$ & $0.10$ & $3$ & $2$ & $9.5867$ & $5.3781$ \\
        $$ & $$ & $5$ & $2$ & $11.7121$ & $7.7281$ \\
        $$ & $$ & $5$ & $3$ & $7.5781$ & $5.1947$ \\
        $$ & $$ & $10$ & $3$ & $10.6632$ & $8.4974$ \\ \hline

   \end{tabular}
    \caption{Values of $E(T)$ for components with negative binomial lifetimes}
    \label{tab2}
    \end{minipage}
\end{table}

\begin{table}[h]
    \begin{minipage}{.45\textwidth}
    \centering
    \begin{tabular}{cccccc}
        \hline
          $q$ & $q_{z}$ & $n$ & $k$ & $E(T)$ & $E(X_{n-k+1:n})$ \\ \hline

            $0.75$ & $0.75$ & $3$ & $2$ & $1.6126$ & $1.0971$ \\
            $$ & $$ & $5$ & $2$ & $1.9946$ & $1.5390$ \\
            $$ & $$ & $5$ & $3$ & $1.4104$ & $1.0857$ \\
            $$ & $$ & $10$ & $3$ & $1.9535$ & $1.6935$ \\ [6pt]
            $0.75$ & $0.90$ & $3$ & $2$ & $1.7692$ & $1.0971$ \\
            $$ & $$ & $5$ & $2$ & $2.1215$ & $1.5390$ \\
            $$ & $$ & $5$ & $3$ & $1.4849$ & $1.0857$ \\
            $$ & $$ & $10$ & $3$ & $2.0096$ & $1.6935$ \\ \hline

   \end{tabular}
    \caption{Values of $E(T)$ for components with discrete Weibull lifetimes}
    \label{tab3}
    \end{minipage}
    \hspace{0.05\textwidth}
    \begin{minipage}{.45\textwidth}
    \centering
    \begin{tabular}{cccccc}
        \hline
          $q$ & $g$ & $n$ & $k$ & $E(T)$ & $E(X_{n-k+1:n})$ \\ \hline

            $0.75$ & $0.25$ & $3$ & $2$ & $1.6629$ & $1.0971$ \\
            $$ & $$ & $5$ & $2$ & $2.0278$ & $1.5390$ \\
            $$ & $$ & $5$ & $3$ & $1.4190$ & $1.0857$ \\
            $$ & $$ & $10$ & $3$ & $1.9572$ & $1.6935$ \\ [6pt]
            $0.75$ & $0.10$ & $3$ & $2$ & $1.8049$ & $1.0971$ \\
            $$ & $$ & $5$ & $2$ & $2.1443$ & $1.5390$ \\
            $$ & $$ & $5$ & $3$ & $1.4905$ & $1.0857$ \\
            $$ & $$ & $10$ & $3$ & $2.0119$ & $1.6935$ \\ \hline

   \end{tabular}
    \caption{Values of $E(T)$ for components with discrete Weibull and geometric lifetimes}
    \label{tab4}
    \end{minipage}
\end{table}

By substituting $r_{1}=1$ and $r_{2}=2$ in \eqref{stNB}, we get $X \le_{st} Y$ for $X \sim ge(p)$ and $Y \sim NB(2,p)$. Thus, as Theorem \ref{s5th1} states, the values of $E\,T$ in Table~\ref{tab2} are larger than the corresponding ones in Table~\ref{tab1}. From \eqref{stW} and the fact that $X \sim W(q,1)$ has a geometric distribution with parameter $1-q$, it follows that $Y \le_{st} X$ for $X \sim ge(p)$ and $Y \sim W(1-p,2)$. Therefore, from Theorem \ref{s5th1} we have that the values of $E\,T$ in Table~\ref{tab1} are larger than the corresponding values in Table~\ref{tab4}, which on the other hand are larger than the corresponding values of $E\,T$ in Table~\ref{tab3}. By using Lemma~\ref{propUSO} and the fact that the order statistic $X_{n-k+1:n}$ is componentwise increasing in $X_{1},\ldots,X_{n}$, we conclude that there are similar relations for the expected lifetimes of the system without a standby unit, which are given in Tables \ref{tab1} - \ref{tab4} as $E(X_{n-k+1:n})$. Since the systems presented in Tables \ref{tab3} and \ref{tab4} differ only on the distribution of the standby unit lifetime, the corresponding values of $E(X_{n-k+1:n})$ are equal.

The example shown in Figure \ref{fig1} is based on a $2$-out-of-$4$ system with a standby unit such that the lifetimes of the four components ($X_1$,$X_2$,$X_3$,$X_4$) and the standby unit ($Z$) are independent and geometrically distributed with parameters $p_{1}=\frac{1}{2}$, $p_{2}=\frac{1}{3}$, $p_{3}=\frac{1}{4}$, $p_{4}=\frac{1}{5}$ and $g=\frac{1}{10}$, respectively. The usual mean residual life function, $E(T-t|T>t)$, the mean residual life function at the system level, $E(T-t|X_{1:n}>t)$, and the mean residual lifetime given that the  $k$-out-of-$n$  system is working, $E(T-t|X_{n-k+1:n}>t)$, are plotted for $0\leq t\leq30$. As Theorem \ref{geometric} states, $E(T-t|X_{1:n}>t)$ is constant in this case. The same three types of mean residual life functions  for a similar system with $X_1 \sim NB(2,\frac{1}{2})$, $X_2 \sim NB(2,\frac{1}{3})$, $X_3 \sim NB(2,\frac{1}{4})$, $X_4 \sim NB(2,\frac{1}{5})$ and $Z \sim NB(2,\frac{1}{10})$ are presented in Figure \ref{fig2}. Since for this system $X_1,X_2,X_3,X_4$ have IFR distributions, from Theorem~\ref{s5th4} it follows that $E(T-t|X_{1:n}>t)$ is decreasing in $t$. Moreover, by \eqref{hrNB} and Theorem~\ref{s5th2} we have that the mean residual lifetime at the system level plotted in Figure \ref{fig2}, for every $t\geq0$,  is above the constant line $E(T-t|X_{1:n}>t)$ presented in Figure \ref{fig1}.

\begin{figure}[!h]
\centering
\includegraphics[width=1\textwidth]{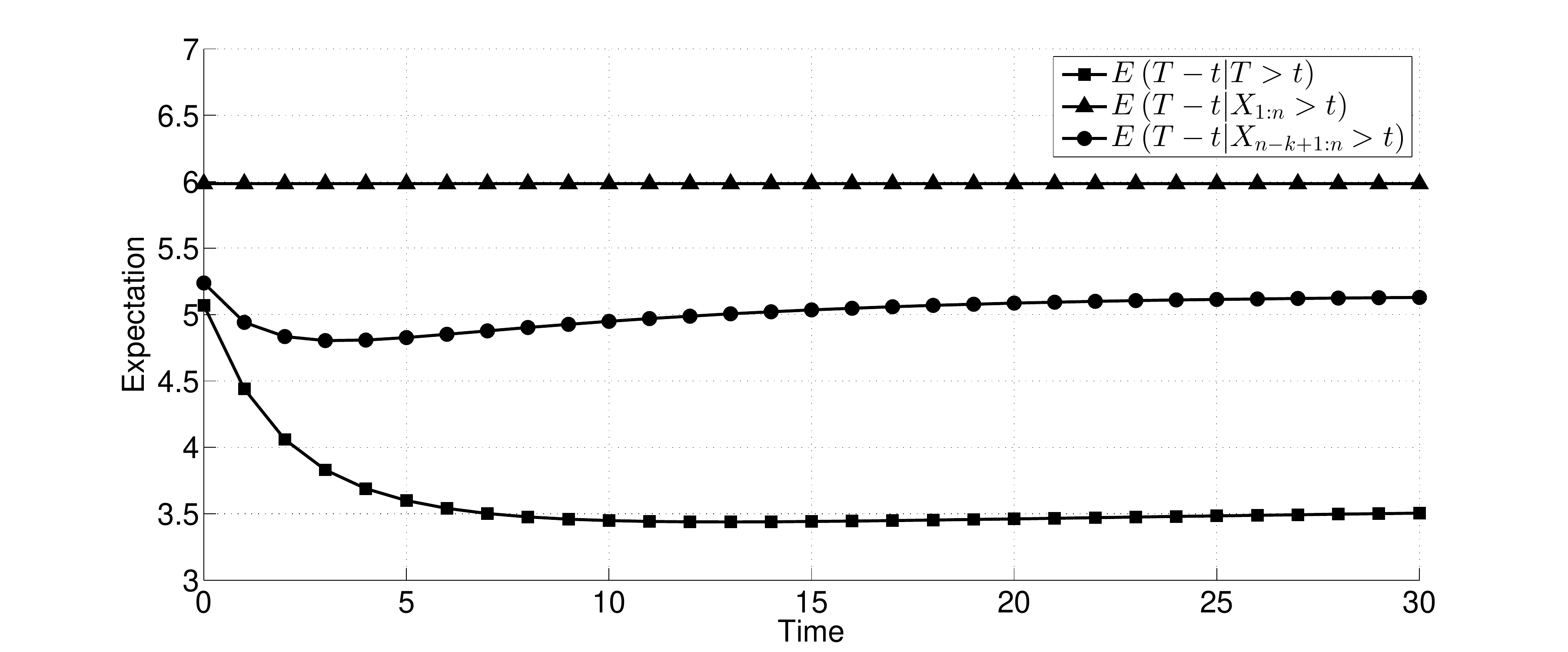}
\caption{Mean residual life functions of a~system with geometric components lifetimes}
\label{fig1}
\end{figure}

\begin{figure}[!h]
\centering
\includegraphics[width=1\textwidth]{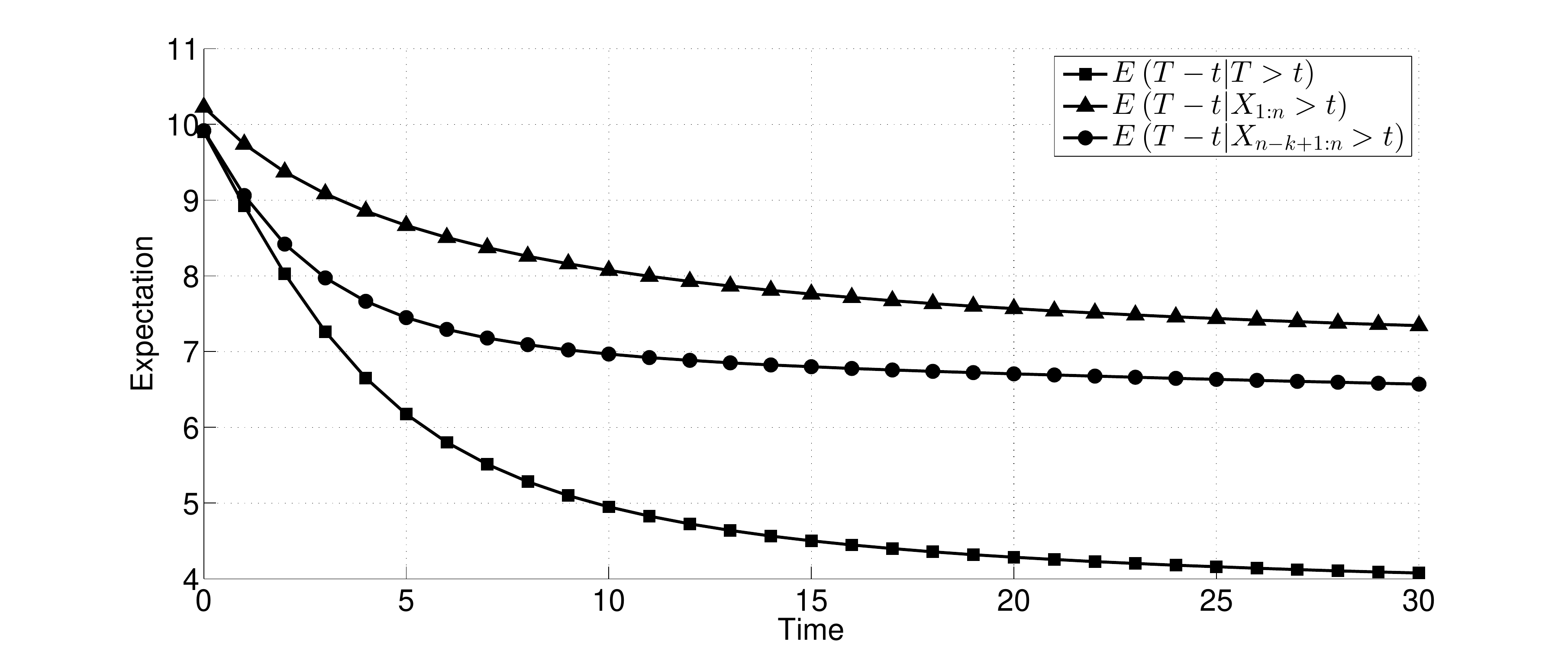}
\caption{Mean residual life functions of a~system with negative binomial components lifetimes}
\label{fig2}
\end{figure}

\begin{figure}[!h]
\begin{minipage}{.45\textwidth}
\centering
\includegraphics[width=1\textwidth]{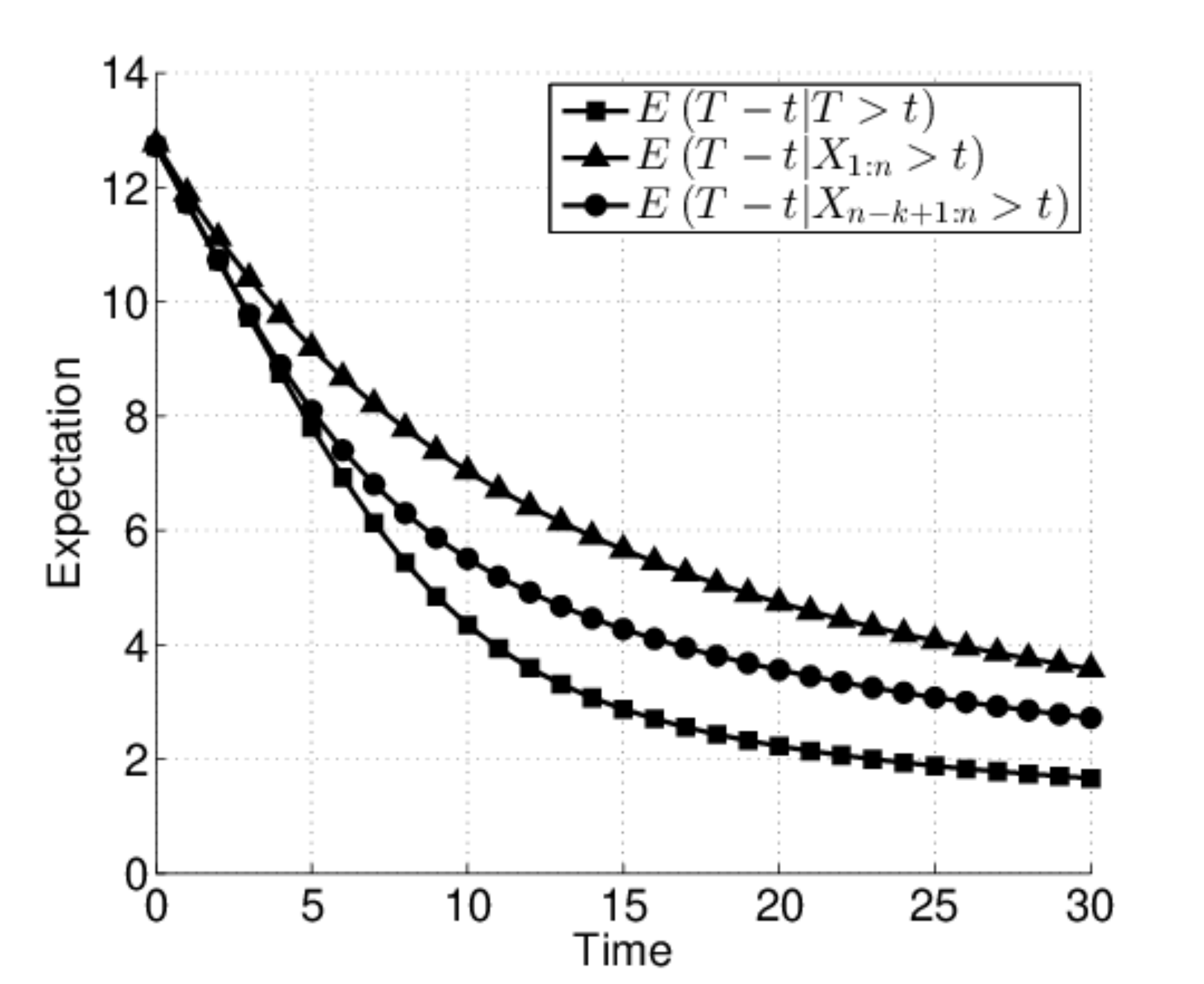}
\caption{Mean residual life functions of a system with discrete Weibull components lifetimes with parameters $q=\exp\left(-\frac{1}{100}\right)$ and $\beta=2$}
\label{fig3}
\end{minipage}
\hspace{0.05\textwidth}
\begin{minipage}{.45\textwidth}
\centering
\includegraphics[width=1\textwidth]{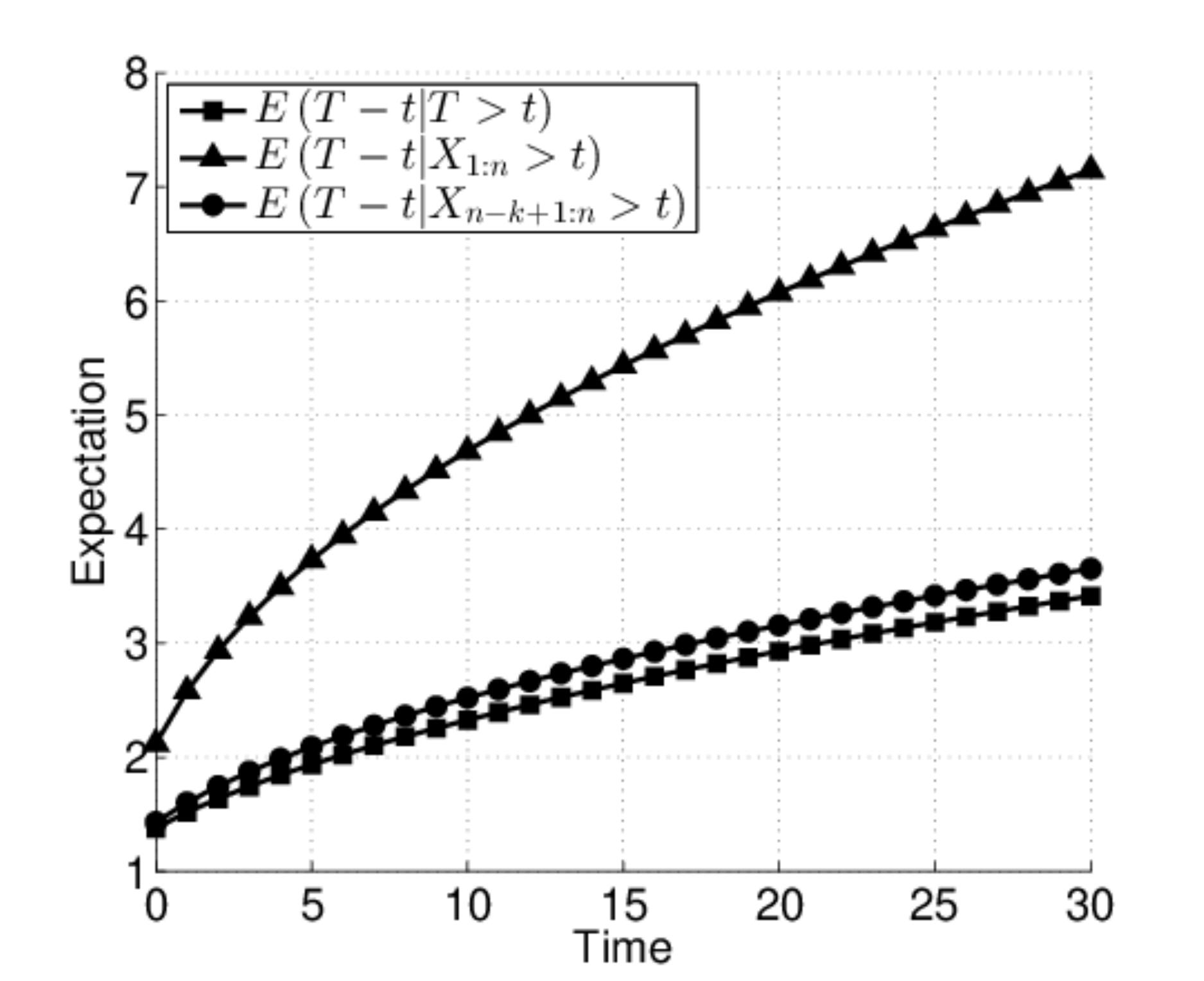}
\caption{Mean residual life functions of a system with discrete Weibull components lifetimes with parameters $q=\exp\left(-2\right)$ and $\beta=\frac{1}{2}$}
\label{fig4}
\end{minipage}
\end{figure}

In Figures \ref{fig3} and \ref{fig4} we plotted the mean residual life functions  of two $2$-out-of-$4$ systems with a standby unit when the lifetimes $(X_{1},X_{2},X_{3},X_{4},Z)$ which are IID and have discrete Weibull distribution. The distribution parameters for the example in Figure \ref{fig3} are $q=\exp\left(-\frac{1}{100}\right)$ and $\beta=2$, while for the example in Figure \ref{fig4} are $q=\exp\left(-2\right)$ and $\beta=\frac{1}{2}$. Since the lifetimes of the components have IFR distributions ($\beta>1$) for the system in Figure~\ref{fig3} and DFR distributions ($\beta<1$) for the system in Figure~\ref{fig4}, from Theorem~\ref{s5th4} it follows that $E(T-t|X_{n-k+1:n}>t)$ is decreasing (increasing) in $t$. As the plots show, it looks like similar results may be also true for the usual mean residual life function and the mean residual lifetime given that the $k$-out-of-$n$ system is working. However, $E(T-t|T>t)$ and $E(T-t|X_{n+k-1:n}>t)$ are not constants in Figure \ref{fig1} (they first decrease and then increase) and hence we conclude that the IFR (DFR) property of the components do not necessarily imply that the functions $E(T-t|T>t)$ and $E(T-t|X_{n-k+1:n}>t)$ are decreasing (increasing) in $t$. Therefore analogues of  Theorem~\ref{s5th4} for the usual residual lifetime nor for the residual lifetime given that the $k$-out-of-$n$ system is working are not true  in  general.

\section{Conclusions}
\label{sec7}

In the present paper, we examine the effect of adding a cold standby unit to a $k$-out-of-$n$ system operating  in discrete time. Expressions for the distribution and expectation of the system lifetime are derived for an arbitrary joint discrete distribution of the components lifetimes. To compute these reliability quantities in practice, we propose a procedure to approximate them with an error not greater than a fixed value $d>0$. This method is applied for systems with geometric, negative binomial and discrete Weibull distributions of the component lifetimes.   Moreover, to describe the aging behavior of $k$-out-of-$n$ systems with a standby unit we consider three types of residual lifetimes: one given that the system is working (called the usual residual lifetime), the second one given that none of its components are broken (called the residual lifetime at the system level) and the third one given that the  $k$-out-of-$n$ system is working. Monotonicity properties of these system characteristics are studied by using the concept of stochastic orders. The obtained results can be applied for comparing different systems not only in the case when components lifetimes have discrete distributions, but in the more general setting of an arbitrary joint distribution (Theorems \ref{s5th1} and \ref{s5th2}) or continuous distributions (Theorem \ref{s5th4}). In particular, Theorem \ref{s5th4} gives positive answer to the question posed in \cite{eryilmaz2012mean} whether the mean residual life function at the system level is decreasing whenever components lifetimes have increasing failure rates. The illustrative examples presented in Section~\ref{ill_exa} show that analogues of Theorem \ref{s5th4} neither  for the usual residual lifetime nor for the residual lifetime given that the $k$-out-of-$n$ system is working are not true in general, i.e. the failure rate properties of the components do not necessarily imply the monotonicity of these two residual life functions. As a~future work, it would be interesting to study  reliability properties of a $k$-out-of-$n$ system with more than one standby unit and check if these results hold in this more general framework.

\section*{Acknowledgments}
A.D. wishes to thank Institute of Mathematics and Informatics, Bulgarian
Academy of Sciences, where the paper was written, for the invitation and
hospitality.  She also
 acknowledges financial support for this research from  Polish National Science Center under Grant no. 2015/19/B/ST1/03100.  The work of N.N. and
E.S. was supported by the National Science Fund of Bulgaria under Grant
DH02-13.


\appendix
\bigskip
\bigskip

\noindent
{\Large \bf Appendices}

\section{Properties of order statistics}
\label{sec:App1}
In this section we derive some new properties of order statistics that are needed in our developments. The first result describes the conditional distribution of the random vector $(X_1-t, \ldots, X_n-t, Z)$ given $X_{1:n}>t$, where $X_1, \ldots, X_n, Z$ are RVs with any joint distribution, not necessary discrete or continuous.

\begin{lemma}\label{sec2:lemma1}
Let $t\in\mathbb{R}$ and the random vector $(X_1, \ldots, X_n, Z)$ have any distribution such that $P(X_1> t, \ldots, X_n> t)>0$. Then the conditional distribution of $(X_1-t, \ldots, X_n-t, Z)$ given $X_{1:n}>t$ is the same as the unconditional distribution of the random vector $(X_1^{(t)}, \ldots, X_n^{(t)}, Z^{(t)})$ with the  reliability function defined in  (\ref{s2f1}).
\end{lemma}

\begin{proof}
It is clear that
\begin{eqnarray*}
  \lefteqn{P(X_1-t>x_1, \ldots, X_n-t>x_n, Z>z|X_{1:n}>t)} \\
 && =P(X_1-t>x_1, \ldots, X_n-t>x_n, Z>z|X_1>t,\ldots, X_n>t)
\end{eqnarray*}
and this conditional probability is equal to (\ref{s2f1}). Hence the lemma follows.
\end{proof}

Lemma \ref{sec2:lemma1} immediately implies the following result.

\begin{corollary}\label{sec2:cor1}
   If $t\in\mathbb{R}$,  $X_1, \ldots, X_n, Z$ are independent RVs satisfying $P(X_i>t)>0$, $i=1,\ldots, n$, and $\bar{F}_i$ denotes the reliability function of $X_i$, $i=1,\ldots, n$, then the conditional distribution of $(X_1-t, \ldots, X_n-t, Z)$ given $X_{1:n}>t$ is the same as the unconditional distribution of  $(X_1^{(t)}, \ldots, X_n^{(t)}, Z)$, where $X_1^{(t)}, \ldots, X_n^{(t)}, Z$ are independent and $X_i^{(t)}$, $i=1,\ldots, n$, has the reliability function given by
   \begin{equation}\label{s2f2}
     P(X_i^{(t)}>x_i) = P(X_i>x_i+t|X_i>t) = \left\{
     \begin{array}{ll} \displaystyle
       \frac{\bar{F}_i(x_i+t)}{\bar{F}_i(t)} & \mbox{if } x_i>0, \\
       1 & \mbox{otherwise}.
     \end{array}\right.
   \end{equation}

\end{corollary}

In the next lemma we restrict our attention to the case when the joint distribution of the RVs $X_1, \ldots, X_n, Z$ is discrete.


\begin{lemma}\label{sec2:lemma2}
Let $X_1, \ldots, X_n, Z$ have any joint discrete distribution, $F_i(x) = P(X_i \le x)$,  \linebreak  $\bar{F}_i(x) = P(X_i > x)$ and $p_i(x) = P(X_i = x)$ be the marginal cumulative distribution function, reliability function and probability mass function of $X_i$, respectively, $F_i(x^-) = P(X_i < x)$ and $\bar{G}(x) = P(Z > x)$. Define the function $ h_{k,n}(\cdot,\cdot)$ by (\ref{s2f3}).
Then, for $u\le t$,
\begin{equation}\label{s2f6}
h_{k,n}(t,u)=\sum_{v=0}^{n-k} \sum_{(j_1,\ldots,j_n) \in {\cal P}_{v,n-k+1}} P\left(\leftidx{_{(X_1, \ldots, X_n, Z)}^{(j_1,\ldots,j_n)}}A^{u,t}_{v,n-k+1}\right)
\end{equation}
\begin{numcases}{=}
\bar{G}(t-u) \sum_{v=0}^{n-k} \sum_{(j_1,\ldots,j_n) \in {\cal P}_{v,n-k+1}} \left(\prod_{\ell=1}^v F_{j_\ell}(u^-) \right) \left(\prod_{\ell=v+1}^{n-k+1} p_{j_\ell}(u)  \right) \left(\prod_{\ell=n-k+2}^n \bar{F}_{j_\ell}(t) \right)   \nonumber\\
 \qquad \qquad  \hbox{ if $X_1,\ldots,X_n,Z$ are independent,} \label{s2f6I}   \\
\frac{n!}{(k-1)!}\sum_{v=0}^{n-k} \frac{1}{v! (n-k+1-v)!} P\left(\leftidx{_{(X_1, \ldots, X_n, Z)}^{(1,\ldots,n)}} A_{v,n-k+1}^{u,t}\right)    \nonumber \\
 \qquad \qquad \hbox{ if $(X_1,\ldots,X_n,Z)$ is exchangeable,}  \label{s2f6E}
\end{numcases}
where  the events $ \leftidx{_{(X_1, \ldots, X_n, Z)}^ {(j_1,\ldots,j_n)}} A^{u,t}_{v,n-k+1}$  are defined in (\ref{defA}).

Moreover, in the IID case defined in (\ref{IIDcase}), formulas (\ref{s2f6I}) and (\ref{s2f6E}) reduce to
\begin{equation}\label{s2f6IID}
 h_{k,n}(t,u)=\frac{n!}{(k-1)!}  (\bar{F}(t))^{k-1} \bar{G}(t-u) \sum_{v=0}^{n-k} \frac{1}{v!(n-k+1-v)!} (F(u^-))^v (p(u))^{n-k+1-v}, \ u\le t.
\end{equation}

\end{lemma}

\begin{proof}
 To show  (\ref{s2f6}) we observe that, for $u\le t$,
\[
  h_{k,n}(t,u)= P\left(\bigcup_{v=0}^{n-k} A_v \right),
\]
where
\begin{eqnarray*}
A_v&=&\{\text{exactly $v$ of $X_i$ are less than $u$},  \\
&&\text{exactly $n-k+1-v$ of $X_i$ are equal to $u$,}  \\
&&\text{the remaining $k-1$ of $X_i$ are greater than $t$, and  $Z>t-u$} \}.
\end{eqnarray*}

The events $A_v$, $v=0,1,\ldots,n-k$, are pairwise disjoint, which yields
\begin{equation} \label{eq:s2f6}
h_{k,n}(t,u)=\sum_{v=0}^{n-k} P(A_v).
\end{equation}
Since
\begin{eqnarray}
 P(A_v)&=&\sum_{(j_1,\ldots,j_n) \in {\cal P}_{v,n-k+1}} P\left(\left(\bigcap_{l=1}^v \{X_{j_l}<u \}\right)\cap \left(\bigcap_{l=v+1}^{n-k+1} \{ X_{j_l} =u \}\right) \right.  \nonumber\\
   &&\left. \cap \left(\bigcap_{l=n-k+2}^n \{X_{j_l} >t \}\right)  \cap \{Z >t-u \}    \right),\label{s2f4}
\end{eqnarray}
(\ref{s2f6}) follows. When $X_1, \ldots, X_n, Z$ are independent, the probabilities in (\ref{s2f4}) factor appropriately   giving (\ref{s2f6I}). If in turn $(X_1, \ldots, X_n, Z)$ is exchangeable, then all the events $_{(X_1, \ldots, X_n, Z)}^{(j_1,\ldots,j_n)} A_{v,n-k+1}^{u,t}$, $(j_1,\ldots, j_n)\in {\cal P}_{v,n-k+1}$, have the same probability. Noting that there are exactly $ \frac{n!}{v!(n-k+1-v)!(k-1)!}$ permutations in ${\cal P}_{v,n-k+1}$ yields (\ref{s2f6E}).
\end{proof}

The following lemma was given in \cite{davies2018computing} for the case when $X_1, \ldots, X_n$ have joint discrete distribution, but an analysis of its proof shows that it holds for any joint distribution of $X_1, \ldots, X_n$.

\begin{lemma} \label{sec2:Lemma3}
For the random vector $X_1,\ldots, X_n$ with any   distribution and $1\le k \le n$, we have
\begin{equation}\label{s2f10}
P(X_{n-k+1:n}>t)=\sum_{v=0}^{n-k}  \sum_{(j_1,\ldots,j_n) \in {\cal P}_{v}} P(\leftidx{_{(X_1,\ldots, X_n)}^{(j_1,\ldots,j_n)}} C_v^t) \hspace{60mm}
\end{equation}
\begin{numcases}{=}
\sum_{s=0}^{n-k} \sum_{(j_1,\ldots,j_n) \in {\cal P}_{v}} \left(\prod_{\ell=1}^v F_{j_\ell}(t) \right)  \left(\prod_{\ell=v+1}^n \bar{F}_{j_\ell}(t) \right) \nonumber   \\
 \qquad \qquad \qquad  \hbox{ if $X_1,\ldots,X_n$ are independent,}  \label{s2f10I} \\
\sum_{v=0}^{n-k}\binom{n}{v} P\left(\leftidx{_{(X_1,\ldots, X_n)}^{(1,\ldots,n)}} C_v^t\right)   \hbox{ if $(X_1,\ldots,X_n)$ is exchangeable,} \label{s2f10E}
\end{numcases}
where  $F_i (t) = P(X_i \le t)$, $\bar{F}_i (t) = P(X_i > t)$ and  the events $\leftidx{_{(X_1, \ldots, X_n)}^{(j_1, \ldots,j_n)}} C_{v}^t$ are defined in (\ref{defC}).
\end{lemma}

\section{Proofs}
\label{sec:App2}

\noindent {\it Proof of Theorem \ref{ETfinite}:}
From (\ref{s2f10}) we have, for $1 \le k \le n$,
  \begin{eqnarray}
   P(X_{n-k+1:n}>t) &=& \sum_{v=0}^{n-k}   \sum_{(j_1,\ldots,j_n) \in {\cal P}_{v}}
 P\left(\left(\bigcap_{\ell=1}^v \{X_{j_\ell} \leq t\} \right) \cap \left( \bigcap_{\ell=v+1}^n \{X_{j_\ell}>t \}\right)\right) \nonumber\\
  &\le&\sum_{v=0}^{n-k}   \sum_{(j_1,\ldots,j_n) \in {\cal P}_{v}} P(X_{jn}>t) \\
 &\le& \sum_{v=0}^{n-k}   \sum_{(j_1,\ldots,j_n) \in {\cal P}_{v}}  \max_{i=1, \ldots,n} P(X_i>t) = P(Y>t) \sum_{v=0}^{n-k} \binom{n}{v},\label{s3f8}
\end{eqnarray}
because there are $\binom{n}{v}$ permutations in ${\cal P}_{v}$.

If $E\,Y<\infty$, then
\[
\sum_{t=0}^{\infty}  P(Y>t)\sum_{v=0}^{n-k} \binom{n}{v} = \sum_{v=0}^{n-k} \binom{n}{v}  E\,Y < \infty,
\]
which by (\ref{s3f8}) gives
\[
E\, X_{n-k+1:n} = \sum_{t=0}^{\infty} P(X_{n-k+1:n}>t) < \infty.
\]
Thus we have proved that
\begin{equation}\label{s3f9}
   E\,Y < \infty \mbox{ implies } E\,X_{n-k+1:n}< \infty \mbox{ for } 1 \le k \le n.
\end{equation}

Now let $k=1$. Then  (\ref{s3f5}) shows that $ET$ is finite if and only if $EZ<\infty$ and $ E\,X_{n:n}< \infty$. From  (\ref{s3f9}) $EX_{n:n}< \infty$ if $ E\,Y< \infty$ and part (a)  follows.

If in turn $2 \le k \le n$, then (\ref{f1}) yields
\begin{equation}\label{s3f10}
T\le X_{n-k+2:n}
\end{equation}
and consequently $E\,T$ is finite if $E\,X_{n-k+2:n}$ is finite. Application of  (\ref{s3f9}) gives  part (b).

In the case when $X_1, \ldots, X_n$ are identically distributed we see that
\[
P(Y>y) = \max_{i=1,\ldots,n} P(X_i>y) =  P(X_1>y)
\]
and hence that $E\,Y=E\, X_1$, which proves the last part of the theorem.
$\hfill \qedsymbol$\\

\noindent {\it Proof of Theorem \ref{calculateET}:}
Since
\[
E\,T =  \sum_{t=0}^{t_0} P(T>t) +  \sum_{t=t_0+1}^{\infty} P(T>t),
\]
(\ref{appET})  approximates $ET$ with an error not greater than $d$ if
\begin{equation}\label{s3f11}
\sum_{t=t_0+1}^{\infty} P(T>t) \le d.
\end{equation}
But, for $t=0,1, \ldots,$
\begin{equation}\label{s3f12}
P(T>t) \le P(X_{n-k+2:n}>t) \le  \max_{i=1,\ldots,n} \bar{F}_i(t)  \sum_{v=0}^{n-k+1} \binom{n}{v}
\end{equation}
by (\ref{s3f10}) and (\ref{s3f8}). Therefore (\ref{s3f11}) is satisfied when (\ref{condt0}) holds.

Under the additional assumption  that $X_1, \ldots, X_n$ are independent, (\ref{s2f10I}) and analysis similar to that in the proof of (\ref{s3f8}) allow us to replace (\ref{s3f12}) by
 \begin{eqnarray*}
   P(T>t) &\le&  P(X_{n-k+2:n}>t) \le \sum_{v=0}^{n-k+1}   \sum_{(j_1,\ldots,j_n) \in {\cal P}_{v}} \prod_{\ell=v+1}^{n} \bar{F}_{j\ell} (t)\\
     &\le&   \sum_{v=0}^{n-k+1}   \sum_{(j_1,\ldots,j_n) \in {\cal P}_{v}}  \left( \max_{i=1, \ldots,n} \bar{F}_{i} (t)  \right)^{n-v} \le   \left( \max_{i=1, \ldots,n} \bar{F}_{i} (t)  \right)^{k-1} \sum_{v=0}^{n-k+1} \binom{n}{v}.
 \end{eqnarray*}
Consequently
$$
  \sum_{t=t_0+1}^{\infty} P(T>t) \le    \sum_{v=0}^{n-k+1} \binom{n}{v} \sum_{t=t_0+1}^{\infty} \left( \max_{i=1, \ldots,n} \bar{F}_{i} (t)  \right)^{k-1}
    \le     \sum_{v=0}^{n-k+1} \binom{n}{v} \left( \sum_{t=t_0+1}^{\infty} \max_{i=1, \ldots,n} \bar{F}_{i} (t)  \right)^{k-1}
$$
  and it follows that (\ref{condt0I}) implies (\ref{s3f11}).

Clearly if $\bar{F}(t) = P(X_i>t) =\bar{F}_i(t)$ for $i=1,\ldots, n$, then $\displaystyle  \max_{i=1, \ldots,n} \bar{F}_{i} (t) = \bar{F} (t)$, which gives (\ref{condt0E}) and (\ref{condt0IID}).

We finish the proof with an observation that the assumption $E\,Y<\infty$ guarantees the existence of $t_0$ satisfying  (\ref{condt0}) -- (\ref{condt0IID}). Indeed, (\ref{defRVY}) and $E\,Y<\infty$ give
\[
 \sum_{t=0}^{\infty} P(Y>t) =  \sum_{t=0}^{\infty}  \max_{i=1, \ldots,n} \bar{F}_{i} (t) < \infty,
\]
which implies $\displaystyle\lim_{t_0 \to \infty} \sum_{t=t_0+1}^{\infty} \max_{i=1, \ldots,n} \bar{F}_{i} (t) =0$.
From this it is evident that we can make the value of $\displaystyle  \sum_{t=t_0+1}^{\infty}  \max_{i=1, \ldots,n} \bar{F}_{i} (t)$ arbitrary small by taking sufficient large $t_0$.
$\hfill \qedsymbol$\\

\noindent {\it Proof of Theorem \ref{calculateETparallel}:}
The proof is much the same as that of Theorem \ref{calculateET}. The difference is that now we write
\[
E\,T = E\,Z + E\, X_{n:n} = E\,Z  + \sum_{t=0}^{t_0} P(X_{n:n}>t) + \sum_{t=t_0+1}^{\infty} P(X_{n:n}>t),
\]
require that  $\displaystyle  \sum_{t=t_0+1}^{\infty} P(X_{n:n}>t) <d$ and use  (\ref{s3f8}) to obtain
\begin{eqnarray}
  P(X_{n:n}>t) &\le &  \max_{i=1, \ldots,n} \bar{F}_{i} (t) \sum_{v=0}^{n-1} \binom{n}{v} = \max_{i=1, \ldots,n} \bar{F}_{i} (t)  \left( \sum_{v=0}^{n}  \binom{n}{v}- \binom{n}{n}  \right) \nonumber\\
   &=&   \max_{i=1, \ldots,n} \bar{F}_{i} (t) \left( 2^n -1  \right). \label{s3f12A}
\end{eqnarray}
$\hfill \qedsymbol$\\

\noindent {\it Proof of Proposition \ref{calculateETURLparallel}:}
If  $k=1$ then $T=X_{n:n} +Z$ and (\ref{t0URL}) can be rewritten  as
\begin{equation}\label{s4.1f1}
 \sum_{s=t_0+1}^{\infty} P(X_{n:n} +Z>s) \le d P(T>t).
\end{equation}
But
\begin{eqnarray*}
  \sum_{s=t_0+1}^{\infty} P(X_{n:n} +Z>s) &\le& \sum_{s=t_0+1}^{\infty} \left(  P(X_{n:n} >\frac{s}{2}) +  P(Z>\frac{s}{2}) \right) \\
    &\le&  2  \sum_{s=[(t_0+1)/2]}^{\infty} \left(  P(X_{n:n} >s) +  P(Z>s) \right),
\end{eqnarray*}
where the first inequality is due to the fact that
\begin{equation}\label{{s4.1f2}}
 \left\{ X_{n:n} +Z>s \right\} \subset \left\{ X_{n:n}>\frac{s}{2} \right\} \cup \left\{Z>\frac{s}{2} \right\}.
\end{equation}
Therefore if (\ref{t0condG}) holds and
\begin{equation}\label{s4.1f12}
    \sum_{s=[(t_0+1)/2]}^{\infty}  P(X_{n:n} >s) \le \frac{1}{4}d P(T>t),
\end{equation}
then (\ref{s4.1f1}) is fulfilled. From (\ref{s3f12A}) we see that (\ref{t0condF}) implies (\ref{s4.1f12}) and the proof is complete.
$\hfill \qedsymbol$\\

\noindent {\it Proof of Theorem \ref{s5th2}:}
 The fact that the hazard rate order implies the usual stochastic order gives
\[
 X_i \le_{st} Y_i, \quad i=1, \ldots, n,
\]
and Lemma \ref{propUSO} ensures that $X_{1:n} \le_{st} Y_{1:n}$. Consequently, for any $t$, $P(X_{1:n}>t) \le P(Y_{1:n}>t)$, and the assumption $P(X_{1:n}>t)>0$ yields $P(Y_{1:n}>t)>0$. From this we see that $P(X_i>t) \ge P(X_{1:n}>t)>0$ and $P(Y_i>t) \ge P(Y_{1:n}>t)>0$, $i=1, \ldots, n$.

Using Theorem \ref{ThNone} we obtain
\begin{equation}\label{s5tf1}
[T_1-t| X_{1:n}>t] \stackrel{d}{=} T_1^{(t)} \; \hbox{ and } \; [T_2-t| Y_{1:n}>t] \stackrel{d}{=} T_2^{(t)},
\end{equation}
where $T_1^{(t)}$ and $T_2^{(t)}$ denote lifetimes of $k$-out-of-$n$ systems with a standby unit each and with independent component lifetimes $X_1^{(t)}, \ldots, X_n^{(t)}, Z_1$ and $Y_1^{(t)}, \ldots, Y_n^{(t)}, Z_2$, respectively.
Moreover, for $ i=1, \ldots, n$ and $s \in \mathbb{R}$,
$$
  P(X_i^{(t)}>s) = P(X_i-t>s| X_i>t) \le   P(Y_i-t>s| Y_i>t)= P(Y_i^{(t)}>s),
$$
where the above inequality is a consequence of Lemma \ref{propHRO}. This shows that $X_i^{(t)} \le_{st} Y_i^{(t)}$, $i=1, \ldots, n$. From Theorem \ref{s5th1} we get
\[
T_1^{(t)} \le_{st}  T_2^{(t)},
\]
which combined with (\ref{s5tf1}) ) gives the assertion of the theorem.
$\hfill \qedsymbol$\\

\noindent {\it Proof of Theorem \ref{s5th3}:}
By $\bar{G}_i$ we denote the reliability function of $Z_i$,  $i=1, 2$.
\begin{description}
  \item[(a)] From (\ref{s3f1})   and assumptions we get, for  $i=1, 2$,

\begin{equation}\label{s5f4}
    P(T_i>t) = \sum_{u=0}^{t} P(X_{n-k+1:n} = u, X_{n-k+2:n} >t) \bar{G}_i(t-u) +P(X_{n-k+1:n}>t).
 \end{equation}
 Hence it is obvious that if $\bar{G}_1(t)  \le \bar{G}_2(t)$ for all $t$, then $P(T_1>t) \le (T_2>t)$ for all $t$. Thus part (a) is proved.

   \item[(b)] Observe that Theorem \ref{ThNone} and assumptions imply
   \[
   P(T_1-t>s| X_{1:n}>t) = P(T_1^{(t)}>s)
 \]
 and
    \[
   P(T_2-t>s| Y_{1:n}>t) = P(T_2^{(t)}>s),
 \]
 where
 \begin{itemize}
 \item $T_1^{(t)}$ is the lifetime of $k$-out-of-$n$ system with one standby unit and with element lifetimes $X_1^{(t)}, \ldots, X_n^{(t)}, Z_1^{(t)}$ having the joint reliability function
  \begin{eqnarray*}
     \lefteqn{P(X_1^{(t)}>x_1, \ldots, X_n^{(t)}>x_n, Z_1^{(t)}>z) }   \\
      &=& \frac{P(X_1>x_1+t, \ldots, X_n>x_n+t)}{P(X_1>t, \ldots, X_n>t)} P(Z_1>z), x_1 \ge 0, \ldots, x_n\ge  0, z \in \mathbb{R};
  \end{eqnarray*}
        \item   $T_2^{(t)}$ is the lifetime of $k$-out-of-$n$ system with one standby unit and with element lifetimes $Y_1^{(t)}, \ldots, Y_n^{(t)}, Z_2^{(t)}$ having the joint reliability function
  \begin{eqnarray*}
     \lefteqn{P(Y_1^{(t)}>y_1, \ldots, Y_n^{(t)}>y_n, Z_2^{(t)}>z) }   \\
      &=& \frac{P(Y_1>y_1+t, \ldots, Y_n>y_n+t)}{P(Y_1>t, \ldots, Y_n>t)} P(Z_2>z), y_1 \ge 0, \ldots, y_n\ge 0, z \in \mathbb{R}.
  \end{eqnarray*}
\end{itemize}
Therefore  we can take $Z_i^{(t)} = Z_i$, $i=1,2$. Moreover, the random vector $(X_1^{(t)}, \ldots, X_n^{(t)})$ and RV $Z_1$ are independent, and the random vector $(Y_1^{(t)}, \ldots, Y_n^{(t)})$ and  RV $Z_2$ are independent. Applying part (a) of the theorem we get $T_1^{(t)} \le_{st}  T_2^{(t)}$ and $E\,T_1^{(t)} \le E\,T_2^{(t)}$ and this is precisely the assertion of part (b).

  \item[(c)] We proceed as in the proof of part (a), the only difference is that we use (\ref{fRLSW1}) rather than (\ref{s3f1}) and replace (\ref{s5f4}) by the following two equalities
 \begin{eqnarray*}
     \lefteqn{P(T_1-t>s| X_{n-k+1:n}>t)}   \\
     &=& \frac{1}{P(X_{n-k+1:n}>t)} \sum_{u=t+1}^{t+s} P(X_{n-k+1:n}=u, X_{n-k+2:n}>t+s) \bar{G}_1(t+s-u) \\
      && + P(X_{n-k+1:n}>t+s| X_{n-k+1:n}>t ) = \psi(\bar{G}_1), say,
 \end{eqnarray*}
  and
  $
  P(T_2-t > s| Y_{n-k+1:n}>t ) =  \psi(\bar{G}_2).
  $
 \end{description}
$\hfill \qedsymbol$\\

\noindent {\it Proof of Theorem \ref{s5th4}:}
     We assume that $X_1, \ldots, X_n$  have IFR distributions. The proof for the DFR case goes along the same lines.

     Applying Theorem \ref{ThNone} we see that, for $t\in {\cal T}$,
  \begin{equation}\label{s5f6}
    [T-t| X_{1:n}>t] \stackrel{d}{=} T^{(t)},
      \end{equation}
  where $T^{(t)}$    is the lifetime of a $k$-out-of-$n$ system with one standby unit and such that
  \begin{itemize}
    \item the lifetimes of active components are $ X_1^{(t)}, \ldots, X_n^{(t)}$ and the lifetime of the standby unit is $Z$;
    \item $ X_1^{(t)}, \ldots, X_n^{(t)},Z$ are independent;
    \item $P(X_i^{(t)}>s) = P(X_i-t>s| X_i>t)$, $i=1,\ldots,n$, $s\in \mathbb{R}$.
  \end{itemize}
The last equality and the IFR property of $X_i$ implies that $P(X_i^{(t)}>s)$ is a decreasing function of   $t\in {\cal T}$. In particular, we have for $t_1<t_2$, $t_1, t_2\in {\cal T}$,
\[
P( X_i^{(t_1)}>s) \ge P( X_i^{(t_2)}>s), \quad  s\in \mathbb{R},
\]
which means that
$
   X_i^{(t_1)} \ge_{st} X_i^{(t_2)}
$
and Theorem \ref{s5th1} gives
\begin{equation}\label{s5f7}
   T^{(t_1)} \ge_{st}T^{(t_2)}.
\end{equation}
Combining (\ref{s5f7}) and (\ref{s5f6}) yields (\ref{s5f5})  and the proof is complete.
$\hfill \qedsymbol$\\

\section{Computing the expected lifetime}
\label{sec:App3}

Let $X_{1},X_{2},\ldots,X_{n},Z$ be RVs taking values in the set of non-negative integers, $\bar{F}_{i}(x)=P(X_{i}>x)$ for $i=1,2,\ldots,n$ and $\displaystyle \sum_{t=t_{0}+1}^{\infty}\max\limits_{1\leq i \leq n}\bar{F}_{i}(t)<\infty$. Then from Theorem~\ref{calculateET} we have
\begin{equation*}
0\leq   E\,T-\sum_{t=0}^{t_0}P(T>t)\leq d, \; \; \mbox{if} \; \; \sum_{t=t_{0}+1}^{\infty}\max\limits_{1\leq i \leq n}\bar{F}_{i}(t)\leq d\left(\sum_{v=0}^{n-k+1}\binom{n}{v}\right)^{-1},
\end{equation*}
where $2\leq k\leq n$, and the expected lifetime, $E\,T$, is approximated with an error not greater than the fixed value $d>0$. In order to calculate $E\,T$ with accuracy $d$ we need to find the corresponding number of terms $t_{0}$ in the approximation $\displaystyle \sum_{t=0}^{t_0}P(T>t)$. Let $b(t_{0})$ be a decreasing function such that
\begin{equation*}
 \sum_{t=t_{0}+1}^{\infty}\max\limits_{1\leq i \leq n}\bar{F}_{i}(t)\leq b(t_{0}).
\end{equation*}
Then the inequality $$\displaystyle 0\leq E\,T-\sum_{t=0}^{t_0}P(T>t)\leq d$$ holds for
 \begin{equation}\label{upperBound}
 \displaystyle t_{0}\geq b^{-1}\left(d\left(\sum_{v=0}^{n-k+1}\binom{n}{v}\right)^{-1}\right),
 \end{equation}
 where $b^{-1}(\cdot)$ is the inverse function of $b(\cdot)$.

From Propositions~\ref{calculateETURL} - \ref{calculateETRLSWparallel} it follows that by using the bound \eqref{upperBound} we can approximate $E(T-t|T>t)$ and $E(T-t|X_{n-k+1:n}>t)$ in a similar way. Therefore here we focus only on approximations of $E\,T$. In the next three subsections we consider the cases when $X_{1},X_{2},\ldots,X_{n}$ have geometric, negative binomial and discrete Weibull distributions.

\subsection{Geometric distribution}
\label{sec6geo}

\begin{proposition}\label{calculateETge}
Let $X_1, \ldots, X_n,Z$ be RVs taking values in the set of non-negative integers and $X_1, \ldots, X_n$ be independent and geometrically distributed with $X_i \sim ge(p_i)$, $i=1,\ldots,n$. Then, for $2\le k \le n$, the approximate formula
\begin{equation}\label{appETge}
E\,T \approx \sum_{t=0}^{t_0} \sum_{u=0}^{t}  h_{k,n}(t,u) + E\, X_{n-k+1:n},
\end{equation}
where the function $h_{k,n}$ is given by (\ref{s2f6}) and $t_0$ satisfies
\begin{equation}\label{t0geom}
t_0 \ge \log_{\displaystyle (1- \min_{i=1,\ldots,n} p_i)} \sqrt[k-1]{ \frac{d}{\binom{n}{k-1}} \left( 1-  (1-\min_{i=1,\ldots,n} p_i )^{k-1}\right)} - 2,
\end{equation}
gives an accuracy not worse than the fixed value $d>0$.
\end{proposition}

\begin{proof}
From  (\ref{s3f4}) and (\ref{s2f3}) we have
\[
E\,T = \sum_{t=0}^{t_0} \sum_{u=0}^{t}  h_{k,n}(t,u) +  \sum_{t=t_0+1}^{\infty} \sum_{u=0}^{t}  h_{k,n}(t,u) + E\, X_{n-k+1:n}.
\]
Therefore to obtain the desired accuracy $d$ of approximation in (\ref{appETge}) it suffices to choose $t_0$ such that
\begin{equation}\label{s3f13}
 \sum_{t=t_0+1}^{\infty} \sum_{u=0}^{t}  h_{k,n}(t,u) \le d.
\end{equation}
By  (\ref{s2f3})
\begin{eqnarray*}
 \sum_{u=0}^{t}  h_{k,n}(t,u) &\le&  \sum_{u=0}^{t} P( X_{n-k+1:n}=u,  X_{n-k+2:n}>t) =  P( X_{n-k+1:n}\le t,  X_{n-k+2:n}>t) \\
    &=&   P( \{\mbox{exactly $n-k+1$ of $X_i$ are $\le t$ and the remaining $k-1$ of $X_i$ are $>t$}  \} )\\
    &=&  \sum_{(j_1,\ldots,j_n) \in {\cal P}_{n-k+1}}
 P\left(\left(\bigcap_{\ell=1}^{n-k+1} \{X_{j_\ell} \leq t\} \right) \cap \left( \bigcap_{\ell=n-k+2}^n \{X_{j_\ell}>t \}\right)\right)\\
 &\le&   \sum_{(j_1,\ldots,j_n) \in {\cal P}_{n-k+1}} P \left( \bigcap_{\ell=n-k+2}^n \{X_{j_\ell}>t \}\right)  \\
  &=& \sum_{(j_1,\ldots,j_n) \in {\cal P}_{n-k+1}}\prod_{\ell=n-k+2}^{n} P(X_{j\ell}>t) \le \sum_{(j_1,\ldots,j_n) \in {\cal P}_{n-k+1}} \left( \max_{i=1,\ldots,n}P(X_i>t) \right)^{k-1}\\
&=&  \binom{n}{k-1} \left( 1-  \min_{i=1,\ldots,n} p_i\right)^{(t+1)(k-1)},
\end{eqnarray*}
the next-to-last and last equalities being consequences of assumptions concerning independence and geometric distribution of $X_i$'s, respectively. It follows that
\begin{eqnarray*}
  \sum_{t=t_0+1}^{\infty} \sum_{u=0}^{t}  h_{k,n}(t,u) &\le& \binom{n}{k-1}
\sum_{t=t_0+1}^{\infty} \left( 1-  \min_{i=1,\ldots,n} p_i\right)^{(t+1)(k-1)}\\
    &=&   \binom{n}{k-1} \frac{\left( 1- \displaystyle \min_{i=1,\ldots,n} p_i\right)^{(t_0+2)(k-1)}}{1- \left( 1-  \displaystyle\min_{i=1,\ldots,n} p_i\right)^{k-1}},
\end{eqnarray*}
and  hence that (\ref{s3f13}) is satisfied when (\ref{t0geom}) holds.
\end{proof}

\subsection{Negative binomial distribution}
\label{sec6nb}
A RV $Y$ has negative binomial distribution with parameters $r\in\left\{1,2\ldots\right\}$ and $p\in(0,1)$, for short $Y\sim NB(r,p)$, if $$P(Y=t)=\binom{r+t-1}{t}(1-p)^{t}p^{r}, \quad t=0,1,2,\ldots.$$
Denote the corresponding cumulative distribution function by $F_{r,p}(t)=P(Y\leq t)$ and survival function by $\bar{F}_{r,p}(t)=P(Y > t)$. From formula (5.31) in \cite[p.~218]{johnson2005book} it follows that
\begin{equation}\label{NBrightTail}
  \bar{F}_{r,p}(t)=\sum_{s=t+1}^{\infty}\binom{r+s-1}{s}(1-p)^{s}p^{r}=\sum_{s=0}^{r-1}\binom{r+t}{s}p^{s}(1-p)^{r+t-s}=I_{1-p}(t+1,r),
\end{equation}
where $ t=0,1,2,\ldots$ and $I_{1-p}(t+1,r)=\displaystyle \frac{\displaystyle \int_{0}^{1-p}x^{t}(1-x)^{r-1}dx}{\displaystyle \int_{0}^{1}x^{t}(1-x)^{r-1}dx}$ is the regularized incomplete beta function.
In the special case when $r=1$, the random variable $Y\sim NB(1,p)$ has a geometric distribution with parameter $p$. It is easy to check that $Y\sim NB(r,p)$ has IFR distribution for $r>1$, see \cite[p.~189]{nair2018reliability}. Furthermore, from the fact that $\bar{F}_{r,p}(t)=I_{1-p}(t+1,r)$ is increasing in $r$, it follows that
\begin{equation}\label{stNB}
  Y_1 \le_{st} Y_2,
\end{equation}
where $Y_1\sim NB(r_{1},p)$, $Y_2\sim NB(r_{2},p)$ and $r_{1}\leq r_{2}$.

Since
$$\frac{P(Y>t)}{P(X>t)}=\frac{\displaystyle \sum_{s=0}^{r-1}\binom{r+t}{s}p^{s}(1-p)^{r+t-s}}{(1-p)^{t+1}}=\sum_{s=0}^{r-1}\binom{r+t}{s}p^{s}(1-p)^{r-s+1}$$
increases in $t$, we have that
\begin{equation}\label{hrNB}
  X \le_{hr} Y
\end{equation}
for $X\sim ge(p)$ and $Y\sim NB(r,p)$.

\begin{proposition}\label{calculateETnb}
Let $X_1, \ldots, X_n,Z$ be RVs taking values in the set of non-negative integers and $X_1, \ldots, X_n$ be independent and have negative binomial distribution, i.e. $X_{i}\sim NB(r_{i},p_{i})$ for $i=1,\ldots,n$. Then, for $2\le k \le n$, $\displaystyle \sum_{t=0}^{t_0}P(T>t)$ approximates $E\,T$ with an error not greater than $d>0$ for
\begin{equation}\label{t0NB}
  t_{0}\geq \bar{F}^{-1}_{r,1-p}\left(\frac{pd}{\displaystyle r\sum_{v=0}^{n-k+1}\binom{n}{v}}\right)-1,
\end{equation}
where $r=\max\limits_{1\leq i \leq n}r_{i}$, $p=\min\limits_{1\leq i \leq n}p_{i}$ and $\bar{F}^{-1}_{r,1-p}(\cdot)$ is the inverse function of $\bar{F}_{r,1-p}(\cdot)$.
\end{proposition}
\begin{proof}
Since for fixed $t$ the function $\bar{F}_{r,p}(t)=I_{1-p}(t+1,r)$ is increasing in $r$ and decreasing in $p$, we have that
$$\max\limits_{1\leq i \leq n}\bar{F}_{r_{i},p_{i}}(t)\leq\bar{F}_{r,p}(t), \; \; \mbox{for} \; \; t=0,1,2,\ldots.$$
By using \eqref{NBrightTail},
\begin{align*}
  & \sum_{t=t_{0}+1}^{\infty}\max\limits_{1\leq i \leq n}\bar{F}_{r_{i},p_{i}}(t)\leq\sum_{t=t_{0}+1}^{\infty}\bar{F}_{r,p}(t)= \sum_{t=t_{0}+1}^{\infty}\sum_{s=0}^{r-1}\binom{r+t}{s}p^{s}(1-p)^{r+t-s} \\
  & = \sum_{s=0}^{r-1}\sum_{t=t_{0}+1}^{\infty}\binom{r+t}{s}p^{s}(1-p)^{r+t-s}= \sum_{s=0}^{r-1}\frac{1}{p}\sum_{t=t_{0}+1}^{\infty}\binom{r+t}{r+t-s}p^{s+1}(1-p)^{r+t-s}\\
  &=\sum_{s=0}^{r-1}\frac{1}{p}\sum_{u=t_{0}+1+r-s}^{\infty}\binom{u+s}{u}p^{s+1}(1-p)^{u}=\sum_{s=0}^{r-1}\frac{1}{p}\bar{F}_{s+1,1-p}(t_{0}+r-s)\\
  & \leq \sum_{s=0}^{r-1} \frac{1}{p}\bar{F}_{r,1-p}(t_{0}+1)=\frac{r}{p}\bar{F}_{r,1-p}(t_{0}+1).
\end{align*}
Since $\bar{F}_{r,1-p}(\cdot)$ is a decreasing function, we can use $\displaystyle b(t_{0})=\frac{r}{p}\bar{F}_{r,1-p}(t_{0}+1)$ in \eqref{upperBound} to obtain \eqref{t0NB}.
\end{proof}

\subsection{Discrete Weibull distribution}
\label{sec6dW}
Let $Y$ be a continuous RV having a Weibull distribution with parameters $\alpha>0$ and $\beta>0$, i.e. $Y$ have the survival function
$$\bar{F}_{Y}(t)=P(Y>t)=\exp\left(-\left(\frac{t}{\alpha}\right)^{\beta}\right), \;\;\mbox{for} \;\; t\geq0.$$
Nakagawa and Osaki~\cite{nakagawa1975weibull} proposed a discrete analogue of Weibull distribution by taking $X=\left[Y\right]$, the integer part of $Y$. Then the survival function of $X$ takes the form
$$\bar{F}_{X}(t)=P(X>t)=q^{(t+1)^{\beta}}, \;\;\mbox{for} \;\; t=-1,0,1,\ldots,$$
where $q=\displaystyle \exp\left(-\frac{1}{\alpha^{\beta}}\right)$ and $0<q<1$. For short we denote $X\sim W(q,\beta)$. If $\beta=1$, then $X$ has a geometric distribution with parameter $1-q$. When $\beta>1$, $X$ has IFR distribution and $X$ has DFR distribution for $\beta<1$, see \cite[p.~189]{nair2018reliability}. Moreover, since $\bar{F}(t)=q^{(t+1)^{\beta}}$ is decreasing in $\beta$, it follows that
\begin{equation}\label{stW}
  X_{1} \le_{st} X_{2},
\end{equation}
where $X_{1}\sim W(q,\beta_{1})$, $X_{2}\sim W(q,\beta_{2})$ and $\beta_{1} \geq \beta_{2}$.

\begin{proposition}\label{calculateETdW}
Let $X_1, \ldots, X_n,Z$ be RVs taking values in the set of non-negative integers and $X_1, \ldots, X_n$ be independent. Suppose that $X_{i}$ has a discrete Weibull distribution with parameters $0<q_{i}<1$ and $\beta_{i}>0$, i.e. the survival function of $X_{i}$ is $\bar{F}_{i}(t)=P(X_{i}>t)=q_{i}^{(t+1)^{\beta_{i}}}$, for $i=1,2,\ldots,n$. Then, for $2\le k \le n$, $\displaystyle \sum_{t=0}^{t_0}P(T>t)$ approximates $E\,T$ with an error not greater than $d>0$ for
\begin{equation}\label{t0dW1}
  t_{0}\geq \log_{q}\left(\frac{(1-q)d}{\displaystyle \sum_{v=0}^{n-k+1}\binom{n}{v}}\right)-2, \; \mbox{when } \beta\geq1
\end{equation}
and
\begin{equation}\label{t0dW2}
  t_{0}\geq\left(\frac{1}{\log\left(\frac{1}{q}\right)}\Gamma^{-1}\left(\frac{1}{\beta}+1,\frac{dq^{2}\left(\log\left(\frac{1}{q}\right)\right)^{\frac{1}{\beta}+1}}{\displaystyle (1-q)\sum_{v=0}^{n-k+1}\binom{n}{v}}\right)\right)^{\frac{1}{\beta}}-2, \; \mbox{when } \beta<1,
\end{equation}
where $q=\max\limits_{1\leq i \leq n}q_{i}$, $\beta=\min\limits_{1\leq i \leq n}\beta_{i}$ and  $\Gamma^{-1}(\cdot,\cdot)$ is the inverse with respect to the second variable of the upper incomplete gamma function $\displaystyle \Gamma(s,t):=\int_{t}^{\infty}x^{s-1}\exp(-x)dx$.
\end{proposition}

\begin{proof}
Consider the RV $X$ that has a discrete Weibull distribution with parameters $q=\max\limits_{1\leq i \leq n}q_{i}$ and $\beta=\min\limits_{1\leq i \leq n}\beta_{i}$ and survival function $\bar{F}(t)=P(X>t)=q^{(t+1)^{\beta}}$. Then
$$\max\limits_{1\leq i \leq n}\bar{F}_{i}(t)\leq \bar{F}(t) , \;\;\mbox{for} \;\; t=-1,0,1,\ldots$$
Thus,
\begin{equation}\label{WeibullBound}
  \sum_{t=t_{0}+1}^{\infty}\max\limits_{1\leq i \leq n}\bar{F}_{i}(t)\leq\sum_{t=t_{0}+1}^{\infty}\bar{F}(t)= \sum_{t=t_{0}+1}^{\infty} q^{(t+1)^{\beta}}.
\end{equation}

If $\beta\geq1$, then from \eqref{WeibullBound}
$$\sum_{t=t_{0}+1}^{\infty}\max\limits_{1\leq i \leq n}\bar{F}_{i}(t)\leq\sum_{t=t_{0}+1}^{\infty} q^{(t+1)^{\beta}}\leq\sum_{t=t_{0}+1}^{\infty} q^{(t+1)}=\frac{q^{t_{0}+2}}{1-q}.$$
By substituting $\displaystyle b(t_{0})=\frac{q^{t_{0}+2}}{1-q}$ into \eqref{upperBound}, we obtain \eqref{t0dW1}.

If $\beta<1$, then from \eqref{WeibullBound}
\begin{equation}\label{qBound}
  \sum_{t=t_{0}+1}^{\infty}\max\limits_{1\leq i \leq n}\bar{F}_{i}(t)\leq\sum_{t=t_{0}+1}^{\infty} q^{(t+1)^{\beta}}\leq\sum_{t=t_{0}+1}^{\infty} q^{\left[(t+1)^{\beta}\right]}\; ,
\end{equation}
where $[x]$ is the integer part of $x$. Since there are at most $(s+1)^{\frac{1}{\beta}}-s^{\frac{1}{\beta}}$ integers $t$, such that $s\leq (t+1)^{\beta}<s+1$, i.e. $\left[(t+1)^{\beta}\right]=s$, from \eqref{qBound} it follows that
\begin{align}\label{boundExponential}
  \nonumber \sum_{t=t_{0}+1}^{\infty}\max\limits_{1\leq i \leq n}\bar{F}_{i}(t)\leq & \sum_{s=s_{0}}^{\infty}q^{s}\left((s+1)^{\frac{1}{\beta}}-s^{\frac{1}{\beta}}\right)\leq \sum_{s=s_{0}}^{\infty}q^{s}(s+1)^{\frac{1}{\beta}}-\sum_{s=s_{0}+1}^{\infty}q^{s}s^{\frac{1}{\beta}} \\
  = & \sum_{s=s_{0}}^{\infty}(s+1)^{\frac{1}{\beta}}\left(q^s-q^{s+1}\right)=\frac{1-q}{q^2}\sum_{s=s_{0}+1}^{\infty}q^{s+1}s^{\frac{1}{\beta}},
\end{align}
where $s_{0}=\left[(t_{0}+2)^{\beta}\right]$. By using \eqref{boundExponential} and
$$\sum_{s=s_{0}+1}^{\infty}q^{s+1}s^{\frac{1}{\beta}}\leq \int_{s_{0}+1}^{\infty}q^{x}x^{\frac{1}{\beta}}dx
=\left(\log\left(\frac{1}{q}\right)\right)^{-\frac{1}{\beta}-1}\Gamma\left(\frac{1}{\beta}+1,(s_{0}+1)\log\left(\frac{1}{q}\right)\right),$$
where $\displaystyle \Gamma(\cdot,\cdot)$ is the upper incomplete gamma function, we obtain that
$$ \sum_{t=t_{0}+1}^{\infty}\max\limits_{1\leq i \leq n}\bar{F}_{i}(t)\leq\frac{1-q}{q^2}\left(\log\left(\frac{1}{q}\right)\right)^{-\frac{1}{\beta}-1}\Gamma\left(\frac{1}{\beta}+1,\left(\left[(t_{0}+2)^{\beta}\right]+1\right)\log\left(\frac{1}{q}\right)\right),$$
where the right side is a decreasing function with respect to $t_{0}$. Thus, \eqref{t0dW2} follows from \eqref{upperBound} with $$b(t_{0})=\frac{1-q}{q^2}\left(\log\left(\frac{1}{q}\right)\right)^{-\frac{1}{\beta}-1}\Gamma\left(\frac{1}{\beta}+1,\left(\left[(t_{0}+2)^{\beta}\right]+1\right)\log\left(\frac{1}{q}\right)\right).$$

\end{proof}

\begin{remark}
Since the functions $\bar{F}^{-1}_{r,1-p}(\cdot)$ and $\Gamma^{-1}(\cdot,\cdot)$ are implemented in most software packages for analyzing data, formulas \eqref{t0NB} and \eqref{t0dW2} can be used in practice.
\end{remark}


\end{document}